\documentclass[a4paper]{amsart}

\usepackage{amsthm}
\usepackage{amsfonts}
\usepackage{amsmath,amsthm,amssymb}
\usepackage{comment}
\usepackage{fancybox}
\usepackage{epsf}
\usepackage[dvips]{graphicx}
\usepackage[all]{xy}
\usepackage{float}

\newtheorem{theorem}{Theorem}[section]
\newtheorem{lemma}[theorem]{Lemma}
\newtheorem{proposition}[theorem]{Proposition}

\theoremstyle{definition}
\newtheorem{definition}[theorem]{Definition}
\theoremstyle{remark}
\newtheorem{remark}[theorem]{Remark}
\newtheorem{example}[theorem]{Example}

\begin{document} 

\title[The Milnor $\bar{\mu}$ invariants and nanophrases]{The Milnor $\bar{\mu}$ invariants and nanophrases}

\author[Yuka Kotorii]{Yuka Kotorii}

\address{
Department of Mathematics \\
Tokyo Institute of Technology \\
Oh-okayama \\
Meguro \\
Tokyo 152-8551 \\
Japan
}

\email{kotorii.y.aa@m.titech.ac.jp}



\thanks{}

\begin{abstract} 
Two link diagrams are link homotopic if one can be transformed into the other by a sequence of Reidemeister moves and self crossing changes.
Milnor introduced invariants under link homotopy called $\bar{\mu}$. 
Nanophrases, introduced by Turaev, generalize links.   
In this paper, we extend the notion of link homotopy to nanophrases.
We also generalize $\bar{\mu}$ to the set of those nanophrases that correspond to virtual links.    
 
\end{abstract} 

\date{\today}

\maketitle

%%%%%%%%%%%%%%%%%%%%%%%%%%%%%%%%%%%%%%%%%%%%%%%%%%%%%%%%%%%%%%%%%%%%%%%%%%%%%%%%%%%%%%%%%%%%%%%%%%%%%%%%%%%%%%%%%%%%%%%%%%%%%%%%%%
\section{Introduction}
A word is a sequence of symbols, called letters, belonging to a given set $\mathcal{A}$, called alphabet.  
Turaev developed a theory of words based on the analogy with curves on the plane, knots in the 3-sphere, virtual knots, etc in \cite{MR2195447,MR2276346,MR2295606,MR2352565}.  

A Gauss word is a sequence of letters with the condition that any letter appearing in the sequence does so exactly twice.
A Gauss word can be obtained from an oriented virtual knot diagram introduced by Kauffman in \cite{MR1721925}. 
Given a diagram, label the real crossings and pick a base point on the curve somewhere away from the real crossings. 
Starting from the base point, we follow the curve and read off the labels of the crossings as we pass through them. 
When we return to the base point, we will have a sequence of letters in which each label of a real crossing appears exactly twice. 
Thus this sequence is a Gauss word. 
It is natural to introduce combinatorial moves on Gauss words, based on Reidemeister moves on knot diagrams. 
The equivalence relation generated by these moves is called homotopy.
% We then give each letter an information for the diagram.
We then decorate each letter with information from the diagram.
This leads us to the notion of nanoword as introduced by Turaev in \cite{MR2352565}.  
By introducing combinatorial moves on nanowords, the notion of homotopy can be defined for them also. 
From this viewpoint, homotopy of Gauss words is the simplest kind of nanoword homotopy. 
In fact, Gauss words underlying homotopic nanowords are homotopic as Gauss words.

The theory of nanowords can be naturally generalized to the theory of nanophrases just like knot theory does to link theory. 
The purpose of this paper is to develop a weaker homotopy theory of nanophrases,
called $M$-homotopy, which is an analogue of Milnor's link homotopy \cite{MR0071020,MR0092150}.  
By a link homotopy we mean a deformation of one link into another, during which each component of the link is allowed to cross itself,
but no two components are allowed to intersect.
Milnor introduced invariants under link homotopy called $\bar{\mu}$ in \cite{MR0071020,MR0092150}.    
We introduce a self crossing move on nanophrases and the associated $M$-homotopy allowing self crossings. 
The main result stated in Theorem \ref{m.them} establishes an $M$-homotopy invariant of nanophrases corresponding to virtual links
\footnote{The theory of nanophrases contains virtual link theory as a special case. 
Nanophrase theory is more general in that Turaev considered a general notion of "decoration", not just the case inspired by link diagrams.} 
as an extension of Milnor's $\bar{\mu}$ invariants.
Dye and Kauffman have done a similar work in \cite{MR2673693}. 
We will come back to their results in Remark \ref{counter}. 
Also, Kravchenko and Polyak defined an extension of Milnor's $\mu$ invariants to virtual tangles in \cite{MR2775128}.
While $\bar{\mu}$ is an invariant of virtual links, $\mu$ can be viewed as an invariant of virtual string links, but not as an invariant of  virtual links.
There is a natural surjection from virtual string links to virtual links. 
In that sense Kravchenko and Polyak's extension coincides with ours in the case of virtual string links, modulo lower degree invariants.  
However our proof of invariance is considerably different from theirs.

The paper is organized as follows.
In section 2, following Turaev, we give formal definitions of words, phrases and so on.
In section 3, we develop the $M$-homotopy theory for nanophrases and define invariants analogous to Milnor's $\bar{\mu}$. 
In section 4, we compute some examples for $\bar{\mu}$.
In sections 5, 6 and 7, we prove invariance under $M$-homotopy corresponding to self-crossing moves on virtual links.  
In section 8, we show that $\bar{\mu}$ is also an invariant under another $M$-homotopy corresponding to self-crossing moves on welded links.

%%%%%%%%%%%%%%%%%%%%%%%%%%%%%%%%%%%%%%%%%%%%%%%%%%%%%%%%%%%%%%%%
\section{Nanowords and nanophrases}
In this section, following Turaev \cite{MR2195447,MR2276346,MR2295606,MR2352565}, we review formal definitions of words, phrases and so on. 

%%%%%%%%%%%%%%%%%%%%%%%%%%%%%%%%%%%%%%%%%%%%%%%%%%%%%%%%%%%%%%%%
\subsection{Words and phrases}
An {\it alphabet} is a finite set, whose element is called {\it letter}. 
Let $\mathcal{A}$ be an alphabet. 
For any positive integer $m$, let $[m]$ denote the set $\{ 1, 2, \ldots, m\}$.  
Then, a {\it word} on $\mathcal{A}$ of length $m$ is a map
\[ w: [m] \to \mathcal{A}. \]
Simply, we may think of a word $w$ on $\mathcal{A}$ as a finite sequence of letters in $\mathcal{A}$ 
and we will usually write words in this way.  
For example, $ABA$ is a word of length $3$ on $\{ A,B,C \}$, 
where 1 is mapped to $A$, 2 to $B$ and 3 to $A$.
By convention, the empty word of length 0 on any alphabet is denoted by $\emptyset$.  

The {\it concatenation} of two words is defined by writing down the first word and then the second one.  
For example, the concatenation of the words $w = ABC$ and $v = DBBA$ is the word $wv = ABCDBBA$.  

An {\it phrase} on $\mathcal{A}$ is a sequence of words on $\mathcal{A}$.
If the number of component of the phrase is $n$, we call the phrase $n$-component phrase.  
We write phrases as a sequence of words which is separated by `$|$'.  
For example, $A|ABC|C|D$ is a 4-component phrase on $\{ A,B,C,D \}$.  
There is a unique phrase with 0 component which we denote by $\emptyset_P$.  
In this paper, we will regard words as 1-component phrases.

%%%%%%%%%%%%%%%%%%%%%%%%%%%%%%%%%%%%%%%%%%%%%%%%%%%%%%%%%%%%%%%%%%%%%%%%%%%%%%%%%%%%%%%%%%%%%%%%%%%%%%%
\subsection{Nanowords and nanophrases}
Let $\alpha$ be a finite set.  
An $\alpha$-alphabet is an alphabet $\mathcal{A}$ together with an associated map from $\mathcal{A}$ to $\alpha$.  
The map is called a {\it projection}.  
The image of any $A \in \mathcal{A}$ in $\alpha$ will be denoted by $|A|$.
An {\it isomorphism} of $\alpha$-alphabets, $\mathcal{A}_1$ and $\mathcal{A}_2$, is a bijective map  
$f$ from $\mathcal{A}_1$ to $\mathcal{A}_2$ such that $|f(A)|$ is equal to $|A|$ for any letter $A$ in $\mathcal{A}_1$.

A {\it Gauss word}  on $\mathcal{A}$ is a word on $\mathcal{A}$ such that each letter 
in $\mathcal{A}$ appears exactly twice.  
Similarly, a Gauss phrase on $\mathcal{A}$ is a phrase on $\mathcal{A}$ such that the concatenation of the words 
appearing in the phrase is a Gauss word on $\mathcal{A}$. 
By the definition, we can think that a 1-component Gauss phrase is a Gauss word.

A {\it nanoword} over $\alpha$ is a pair $(\mathcal{A}, w)$,  
where the alphabet $\mathcal{A}$ is an $\alpha$-alphabet and the word $w$ is a Gauss word on $\mathcal{A}$.  
An $n$-component {\it nanophrase} over $\alpha$ is a pair $(\mathcal{A}, p)$,  
where $\mathcal{A}$ is an $\alpha$-alphabet and $p$ is an $n$-component Gauss phrase on $\mathcal{A}$.  

Two nanophrases over $\alpha$, $(\mathcal{A}_1, p_1)$ and $(\mathcal{A}_2, p_2)$, are {\it isomorphic} 
if there exists a bijection map $f$ from $\mathcal{A}_1$ to $\mathcal{A}_2$ 
such that $f$ sends letterwise the $i$th component of $p_1$ to the $i$th component of $p_2$ for any $i$.

Rather than writing $(\mathcal{A}, p)$, we will simply use $p$ to indicate a nanophrase.
When we write a nanophrase in this way, we do not forget the set $\mathcal{A}$ of letters and the projection $\mathcal{A}\rightarrow \alpha$. 

%%%%%%%%%%%%%%%%%%%%%%%%%%%%%%%%%%%%%%%%%%%%%%%%%%%%%%%%%%%%%%%%%%%%%%%%%%%%%%%%%
\subsection{Equivalence relations of nanophrases}
Fix a finite set $\alpha $ and then let $\tau$ be an involution on $\alpha$ (that is, $\tau (\tau (a))$ is equal to $a$ for all $a \in \alpha$).  
Let $S$ be a subset of $\alpha \times \alpha \times \alpha$.  
We call the triple $(\alpha , \tau , S)$ a {\it homotopy} data.

Fixing a homotopy data $(\alpha , \tau , S)$, we define three {\it homotopy moves} on nanophrases over $\alpha$ as follows.
In the moves on nanophrases, the lower cases $x, y, z$ and $t$ represent any sequences of letters which is possible to include one or more `$|$', 
so that the phrase on each side of the move is a nanophrase.  
The moves are

\hspace{1.5em}move H1: for any $|A|$,

\hspace{6.5em}$(\mathcal{A}, xAAy) \longleftrightarrow (\mathcal{A}-\{ A\}, xy)$

\hspace{1.5em}move H2: if $\tau (|A|) = |B|$,

\hspace{6.5em}$(\mathcal{A}, xAByBAz) \longleftrightarrow (\mathcal{A}-\{ A, B\} , xyz)$

\hspace{1.5em}move H3: if $(|A|, |B|, |C|) \in S$,

\hspace{6.5em}$(\mathcal{A}, xAByACzBCt) \longleftrightarrow (\mathcal{A}, xBAyCAzCBt)$.  

\begin{figure}[H]
\includegraphics[scale=0.55]{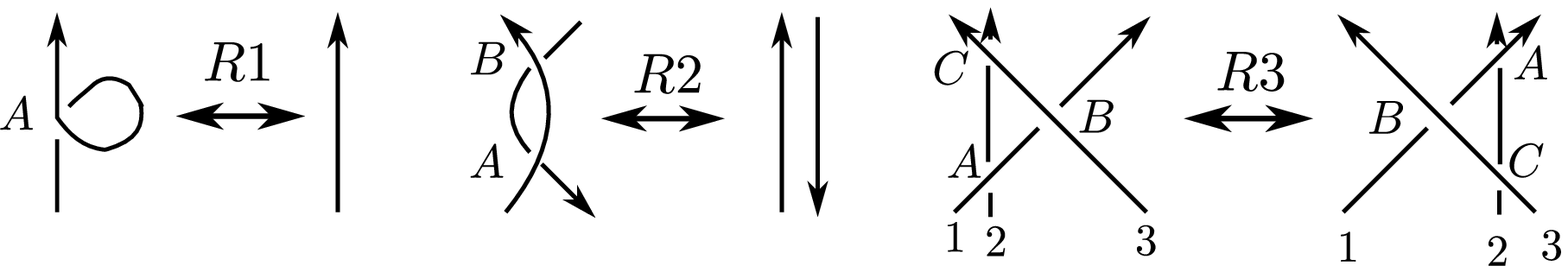}
\caption{}
\label{fig:1}
\end{figure}

The moves H1 - H3 on nanophrases correspond to Reidemeister moves R1 - R3 (in Fig.\ \ref{fig:1}) on link with base points, respectively.
Here we can only move the arcs in Fig.\ \ref{fig:1} if these arcs do not contain the base points.
In other words, no arc can cross the base points of the link.  
% When we consider the correspondence of the moves H1 - H3 and the moves R1 - R3, 
% we must notice the base point of link.  
% Then we can not cross the base point of the link in Fig.\ \ref{fig:1}.
The {\it homotopy} is an equivalence relation of nanophrases over $\alpha $ generated by isomorphisms and the three homotopy moves H1 - H3.
The homotopy depends on the choice of the homotopy data  $(\alpha , \tau , S)$, so different choices of homotopy data provide different equivalence relations.
For the moves H1 - H3, components of any nanophrases were not added and removed, 
and so the number of the components is an invariant under any kind of homotopy.  

In \cite{MR2276346}, Turaev defined a {\it shift move} on nanophrases.
Let $\nu $ be an involution on $\alpha $ which is independent of $\tau $.
Let $p$ be an $n$-component nanophrase over $\alpha $.
A  {\it shift} move on the $i$th component of $p$ is a move which gives a new nanophrase $p' $ as follows.
If the $i$th component of $p$ is empty or contains a single letter, then $p'$ is $p$.
Otherwise, the $i$th component of $p$ has the form $Ax$.
Then the $i$th component of $p'$ is $xA$, and for all $j$ not equal to $i$ the $j$th component of $p'$ is the same as that of $p$.
Furthermore, if we write $|A|_p$ for $|A|$ in $p$ and $|A|_{p'}$ for $|A|$ in $p'$, then $|A|_{p'}$ equals $\nu(|A|_{p})$ 
when $x$ contains the letter $A$ and otherwise, $|A|_{p'}$ equals $|A|_p$.

We also call the equivalence relation generated by isomorphisms, homotopy moves and shift moves a {\it homotopy} of nanophrases over $\alpha $.
We call the first homotopy without shift moves an {\it open homotopy} and the second homotopy simply a {\it homotopy}.   
The homotopy depends on the triple $(\alpha ,\tau ,S)$ and $\nu $.   

Let $\alpha_{v} $ be the set $\{a_+,a_-,b_+,b_-\}$ and $\tau _{v} $ the involution on $\alpha_{v} $ which sends  $a_+$ to $b_-$ and $a_-$ to $b_+$.
Let $S_{v} $ be the set \\

\hspace{5.5em}$S_v  = \left\{
\begin{array}{l}
{(a_+,a_+,a_+),(a_+,a_+,a_-),(a_+,a_-,a_-), }\\  
   {(a_-,a_-,a_-),(a_-,a_-,a_+),(a_-,a_+,a_+),}  \\ 
   {(b_+,b_+,b_+),(b_+,b_+,b_-),(b_+,b_-,b_-), } \\  
  {(b_-,b_-,b_-),(b_-,b_-,b_+),(b_-,b_+,b_+) }
\end{array}
\right\}$ .\\

In \cite{MR2276346}, Turaev proved 
\begin{theorem}[Turaev \cite{MR2276346}] \label{e.thm}
(i)The set of homotopy classes of nanowords over $\alpha_{v} $ under the open homotopy with respect to $(\alpha_{v}  ,\tau_{v}  ,S_{v} )$ is in a bijective correspondence to open virtual knots. 
Moreover under the same homotopy data, the set of homotopy classes of nanophrases over $\alpha_{v}  $ is also in a bijective correspondence to the set of stable equivalence classes of pointed ordered link diagrams on oriented surfaces. 

(ii)Let $\nu _{v} $ be the involution on $\alpha _{v} $, where $a_+$ is mapped to $b_+$ and $a_-$ to $b_-$.
Under the homotopy defined by $(\alpha_{v}  ,\tau_{v}  ,S_{v} )$ and $\nu_v$, the set of homotopy classes of nanophrases over $\alpha_{v}$ is in a bijective correspondence to ordered virtual links. 
\end{theorem}

%%%%%%%%%%%%%%%%%%%%%%%%%%%%%%%%%%%%%%%%%%%%%%%%%%%%%%%%%%%%%%%%%%%%%%%%%%%%%%%%%%%%%%%%%%%%%%%%%%%%%%%%%%%%%%%%%%%%%%%%%%%%%%%%%%%%%%%%%%%%%%%%%%%%
\section{The $\bar{\mu} $ invariants and nanophrases}
In this section, we introduce a self crossing move on nanophrases and a new weak homotopy, by allowing self crossing moves. 
We call this an $M$-homotopy. 
Consider the set of equivalence classes of nanophrases corresponding to ordered virtual links,
that is equivalence classes of nanophrases defined by $(\alpha_{v},\tau_{v},S_{v})$ and $\nu_{v} $. 
We define an $M$-homotopy invariant of nanophrases corresponding to ordered virtual links.
This invariant is an extension of $\bar{\mu}$ invariants introduced by Milnor in \cite{MR0071020,MR0092150}. 

%%%%%%%%%%%%%%%%%%%%%%%%%%%%%%%%%%%%%%%%%%%%%%%%%%%%%%%%%%%%%%%%%%%%%%%%%%%%%%%%%%
\subsection{Self crossing and $M$-homotopy}  
Let ($\alpha ,\tau ,S$) be any homotopy data and $\nu$ an involution on $\alpha$ independent of $\tau$.
We introduce a {\it self crossing move} on nanophrases over $\alpha $. 
Let $\sigma $ be an involution on $\alpha $ independent of $\tau $ and $\nu $.
Let $p$ be an $n$-component nanophrase over $\alpha $.
A {\it self crossing move} on the $k$th component of $p$ is a move  which gives a new nanophrase $p'$ as follows.
If there is a letter $A$ in $\mathcal{A}$ which appears exactly twice in the $k$th component of $p$, 
then the $k$th component of $p$ has the form $xAyAz$. 
Then the $k$th component of $p'$ also has the form $xAyAz$.
Furthermore, writing $|A|_p$ for $|A|$ in $p$ and $|A|_{p'}$ for $|A|$ in $p'$, 
we have the identity $|A|_{p'} = \sigma (|A|_p)$. 

We then define {\it $M$-homotopy} and consider it in this paper.
\begin{definition}
{\it Open $M$-homotopy} is the equivalence relation of nanophrases over $\alpha $ generated by isomorphisms, the three homotopy moves H1 - H3 with respect to $(\alpha, \tau ,S)$ and self crossing moves with respect to $\sigma$. 
Similarly, {\it $M$-homotopy} is the equivalence relation of nanophrases over $\alpha $ generated by isomorphisms, 
three homotopy moves with respect to $(\alpha ,\tau ,S)$, 
self crossing moves with respect to $\sigma$ and shift moves with respect to $\nu$.   
\end{definition}

We give an example.
The linking matrix of a nanophrase was first defined by Fukunaga \cite{label7389}
for some special $S$ and by Gibson \cite{label8454} in general.  
It is an invariant under any kind of homotopy of nanophrases. 
We recall the definition here.
First of all, let $\pi $ be the abelian group generated by elements in $\alpha $ with the relations $a + \tau (a) = 0$ for all $a$ in $\alpha $.
In \cite{label7389}, $\pi $ is written multiplicatively, but here we will write it additively. 
For an $n$-component nanophrase $p$, the {\it linking matrix} of $p$ is defined as follows. 
Let $\mathcal{A}_{ij} (p)$ be the set of letters which have one occurrence in the $i$th component of $p$ and the other occurrence in the $j$th component of $p$. 
Let $l_{ii}(p)$ be $0$, and when $j$ is not equal to $i$, let $l_{ij}(p)$ be
\[l_{ij}(p) =\sum _{A\in \mathcal{A}_{ij} (p)} |A| . \]
The linking matrix $L(p)$ is a symmetric $n\times n$ matrix with entries $l_{ij}(p)$'s in $\pi $.
It is easy to see that the linking matrix of nanophrases is a homotopy and $M$-homotopy invariant of nanophrases.

%%%%%%%%%%%%%%%%%%%%%%%%%%%%%%%%%%%%%%%%%%%%%%%%%%%%%%%%%%%%%%%%%%%%%%%%%%%%%%%%%%%%%%%%%%%%%%%%%%%%%%%%%%%%%%%%%%%%%%%%%%%%%%%%%%%%%%%%%%
\subsection{Definition of $\bar{\mu} $} 
From now on, we work only with the homotopy defined by $(\alpha_{v}  ,\tau_{v}  ,S_{v} )$, 
$\nu_{v} $ and the involution $\sigma_v$ which sends $a_+$ to $a_-$ and $b_+$ to $b_-$.

Let $(\mathcal{A},p)$ be an $n$-component nanophrase whose Gauss phrase $p$ is represented by $w_1|w_2| \cdots |w_n$. 
We will write $p$ for $(\mathcal{A},p)$ for simplicity.
Let $w_i$ be $A_{i1}A_{i2} \cdots A_{im_i}$, where $A_{ij}$'s are letters in $\mathcal{A}$.
For each $w_i$, we define a word on $\mathcal{A} \cup \mathcal{A}^{-1}$ by
\[ w_i^{\varepsilon} = A_{i1}^{\varepsilon _{i1}}A_{i2}^{\varepsilon _{i2}} \cdots A_{im_i}^{\varepsilon _{im_i}}, \] 
where $\varepsilon _{ij}$ is determined as follows.
Since $p$ is a nanophrase, any letter appears exactly twice in the Gauss phrase $w_1|w_2| \cdots |w_n$.
Let $A$ denote the letter represented by $A_{ij}$.
Then there exists a unique pair of integers $(k,l)$ $((k,l) \neq (i,j))$ so that $A_{kl}$ represents $A$.
In other words, the other $A$ appears as the $l$th letter on the $k$th component.    
If $i<k$ and $|A| = b_+$ or $i>k$ and $|A| =a_+$, then $\varepsilon _{ij}=1$.
If $i<k$ and $|A| = a_-$ or $i>k$ and $|A| =b_-$, then $\varepsilon _{ij}=-1$.
Otherwise, $\varepsilon _{ij}=0$. 
Namely, if the letter $A$ appears exactly once in the $i$th component,
and if $A_{ij}$ appears earlier (resp. later) than the other $A$ and $|A| = b_+$ (resp. $a_+$), then $\varepsilon _{ij}$ is $1$. 
If $A$ appears exactly once in the $i$th component, and if $A_{ij}$ appears earlier (resp. later) than the other $A$ and $|A| = a_- $ (resp. $b_-$), 
then $\varepsilon _{ij}$ is $-1$. 
For other cases, let $\varepsilon _{ij}$ be zero. 
Here we will call a letter $A_{ij}$ with sign $\varepsilon _{ij}$ the {\it signed letter} 
and a word $w_i$ with sign the {\it signed $i$th component} of the nanophrase $p$.
We note that given a nanophrase $w=w_1|w_2| \cdots |w_n$, 
each letter used (twice) in $w$ appears at most once in the resulting phrase $w_1^\varepsilon|w_2^\varepsilon| \cdots |w_n^\varepsilon$. \\
In the following, we use the convention that $A^0=\emptyset$ and ${(A_1 A_2 \cdots A_n)}^{-1}=A_n^{-1}A_{n-1}^{-1} \cdots A_1^{-1}$.
We note that $AA^{-1} \neq \emptyset$.

\begin{example} \label{ex2}
Consider a nanophrase $p=AB|CDB|DEA|FFCE $, where $|A|=|B|=|C|=b_+,|D|=b_-,|E|=|F|=a_-$.
Then 
\begin{align*}
w_1^\varepsilon  = AB, \hspace{.3cm}
w_2^\varepsilon  =C,  \hspace{.3cm}
w_3^\varepsilon  =D^{-1}E^{-1}, \hspace{.3cm}
w_4^\varepsilon  =\emptyset. 
\end{align*}
\end{example}

Let $\mathcal{L}$ denote the set of words on $\mathcal{A} \cup \mathcal{A}^{-1}$.
Then we define a sequence of  maps $\rho^q$ $(q=2,3,\cdots)$ from $\mathcal{L}$ to itself by induction on $q$.
\begin{align*}
\rho^2(A_{ij}^\pm ) &= A_{ij}^\pm  \\
\rho^q(A_{ij}^\pm ) &= \rho^{q-1}(x_{ij}^{-1})A_{ij}^\pm \rho^{q-1}(x_{ij}), ~~q\geq 3 \\
\rho^q(\emptyset ) &= \emptyset \text{ for all } q\geq 2, 
\end{align*}
where $x_{ij} = A_{k1}^{\varepsilon _{k1}}A_{k2}^{\varepsilon _{k2}} \cdots A_{kl-1}^{\varepsilon _{kl-1}}$,
that is $x_{ij}$ is the signed word obtained by truncating $w_k^\varepsilon$ at $A_{kl}^{\varepsilon_{kl}}$ 
(where the pair $(k,l)$ is derived from the pair $(i,j)$ as explained above).  
We call the word $\rho ^q(A_{ij}^{\varepsilon_{ij}})$ the {\it expanding word} of $A_{ij}$, and $\rho ^q(w_i^{\varepsilon})$ that of $w_i$.
% We will be concerned exclusively with $\rho ^q(w_i^{\varepsilon})$. 
 
\begin{example}\label{ex2.5}
Consider a nanophrase $p=AB|CDB|DEA|FFCE $, where $|A|=|B|=|C|=b_+,|D|=b_-,|E|=|F|=a_-$ (Example \ref{ex2}).
Then 
\begin{align*}
& \rho^2 (w_1^\varepsilon ) = AB,\hspace{.3cm} \rho^2 (w_2^\varepsilon ) =C,\hspace{.3cm} \rho^2 (w_3^\varepsilon ) =D^{-1}E^{-1}, \hspace{.3cm} \rho^2 (w_4^\varepsilon ) =\emptyset. \\
& \rho^3 (w_1^\varepsilon ) = EDAD^{-1}E^{-1}C^{-1}BC ,\hspace{.3cm} \rho^3 (w_2^\varepsilon ) = C, \\
&\rho^3 (w_3^\varepsilon ) =C^{-1}D^{-1}CE^{-1}, \hspace{.3cm} \rho^3 (w_4^\varepsilon ) = \emptyset . \\
& \rho^4 (w_1^\varepsilon ) = \rho^3 (ED)A\rho^3(D^{-1}E^{-1})\rho^3(C^{-1})B\rho^3(C) \\
& \hspace{1cm} = E\rho^2 (C^{-1})D\rho^2 (C)A\rho^2 (C^{-1})D^{-1}\rho^2 (C)E^{-1}C^{-1}BC, \\
& \hspace{1cm} = EC^{-1}DCAC^{-1}D^{-1}CE^{-1}C^{-1}BC, \\
&\rho^4 (w_2^\varepsilon ) = C, \hspace{.3cm} \rho^4 (w_3^\varepsilon ) =C^{-1}D^{-1}CE^{-1}, \hspace{.3cm} \rho^4 (w_4^\varepsilon ) = \emptyset. \\
\end{align*}
\end{example}

Recall that $p=w_1|w_2| \cdots |w_n$ is an $n$-component nanophrase.
We call the index $i$ of a component $w_i$ the order of $w_i$.
Let $M$ denote a finite set $\{a_1, \ldots , a_n\}$. 
Let $\mathcal{M}$ denote the set of words on $M \cup M^{-1}$.
We define a map $\eta$ from $\mathcal{L}$ to $\mathcal{M}$ as follows.
For any letter $A$ in $\mathcal{A}$, let $\eta(A)$ be $a_k$ and $\eta(A^{-1})$ be $a_k^{-1}$, where $k$ is determined by the following rule.
The map $\eta$ assigns to $A$ the component that $A$ belongs to.
If $|A|=b+$ or $a-$, then $k$ is the order of the component of $p$ in which the second $A$ occurs,
and if $|A|=a+$ or $b-$, then $k$ is the order of the component of $p$ in which the first $A$ occurs.
To see what $k$ is for a letter in $w_i$, let us recall the definition of $w_i^\varepsilon$.
If $|A|=b+$ or $a-$, then the sign of the second $A$ is 0, and if $|A|=a+$ or $b-$, then that of the first $A$ is 0.
Thus, if $A_{ij}$ represents $A$ and $\varepsilon_{ij} = 0$, then $\eta(A)$ is $a_i$.    

\begin{example}\label{ex2.55}
Consider a nanophrase $p=AB|CDB|DEA|FFCE$, where $|A|=|B|=|C|=b_+$, $|D|=b_-$, $|E|=|F|=a_-$ (Example $\ref{ex2.5}$).
Then we have
\begin{align*}
& \eta ( \rho^2 (w_1^\varepsilon )) =a_3 a_2,~~ \eta ( \rho^2 (w_2^\varepsilon ) )={a_4}, ~~\eta ( \rho^2 (w_3^\varepsilon ) ) =a_2^{-1}a_4^{-1},~~ \eta ( \rho^2(w_4^\varepsilon ) ) = \emptyset. \\
& \eta ( \rho^3(w_1^\varepsilon) )= a_4 a_2 a_3a_2^{-1}a_4^{-1}a_4^{-1}a_2 a_4,~~ \eta ( \rho^3 (w_2^\varepsilon) )  = a_4,~~ \\
&\eta ( \rho^3 (w_3^\varepsilon) ) = a_4^{-1}a_2^{-1}a_4a_4^{-1},~~\eta ( \rho^3(w_4^\varepsilon) ) = \emptyset. \\
& \eta ( \rho^4(w_1^\varepsilon) ) = a_4a_4^{-1}a_2 a_4 a_3a_4^{-1}a_2^{-1}a_4a_4^{-1}a_4^{-1}a_2 a_4, ~~\eta ( \rho^4 (w_2^\varepsilon) ) = a_4, \\
& \eta ( \rho^4 (w_3^\varepsilon) ) =a_4^{-1}a_2^{-1}a_4a_4^{-1},~~ \eta ( \rho^4(w_4^\varepsilon) ) = \emptyset. \\
\end{align*}
\end{example}

We define a map $\varphi$ from $\mathcal{M}$ to $\mathbb{Z} [[\kappa_1,\kappa_2, \ldots, \kappa_n]]$ by
\begin{eqnarray*}
\varphi (a_h) &=& 1 + \kappa _h \\
\varphi (a_h^{-1}) &=& 1 - {\kappa _h} + \kappa _h^2 - \kappa _h^3 + \cdots, 
\end{eqnarray*}
where $\mathbb{Z} [[\kappa_1,\kappa_2, \ldots, \kappa_n]]$ is the ring of formal power series on non-commuting variables $\kappa_1, \kappa_2, \ldots, \kappa_n$. 

We consider $\varphi \circ \eta ( \rho ^q(w_i^\varepsilon))$ in $\mathbb{Z} [[\kappa_1,\kappa_2, \ldots, \kappa_n]]$.
Since $\varphi \circ \eta ( \rho ^q(w_i^\varepsilon))$ agrees with $\varphi \circ \eta ( \rho ^r(w_i^\varepsilon))$ for any $r\geq q$ up to degree $q$,
the coefficient of a term $\kappa _{c_1} \kappa _{c_2}  \ldots \kappa _{c_u}$ in $\varphi \circ \eta ( \rho ^q(w_i^\varepsilon))$ converges as $q\rightarrow \infty $.
Thus we have a well-defined expansion,
\[ \lim_{q\rightarrow \infty} \varphi \circ \eta ( \rho ^q(w_i^\varepsilon) )= 1+ \sum \mu (p; c_1, c_2, \ldots ,c_u,i) \kappa _{c_1} \kappa _{c_2}  \cdots \kappa _{c_u}. \]
where $c_1, c_2, \ldots ,c_u,i$ is a sequence of integers  between $1$ and $n$. 
Here we note that the integers in the sequence are not necessarily mutually different. 

Let $\Delta (p; c_1, c_2, \ldots ,c_u,i)$ denote the greatest common divisor of $\mu (p;d_1, d_2, \ldots ,d_t)$, 
where the sequence $d_1, d_2, \ldots ,d_t$ $(2\leq t \leq u)$ ranges over all sequences obtained by eliminating at least one of $c_1, c_2, \ldots ,c_u,i$ and permuting the remaining indices cyclically.
We also define $\Delta (p; c_1,i) = 0 $.
Let $\bar{\mu} (p; c_1, c_2, \ldots ,c_u,i)$ denote the residue class of $\mu (p; c_1, c_2, \ldots ,c_u,i)$ modulo $\Delta (p; c_1, c_2, \ldots ,c_u,i)$. 

The main theorem of this paper is as follows.
\begin{theorem} \label{m.them}

Let $p$ be an $n$-component nanophrase.
Let $c_1, c_2, \ldots ,c_u,i$ be a sequence of integers between $1$ and $n$ such that $c_1, c_2, \ldots ,c_u,i$ are pairwise distinct.
Then $\bar{\mu} (p; c_1, c_2, \ldots ,c_u,i)$ 
is an invariant under $M$-homotopy of nanophrases with respect to $(\alpha_v, \tau_v,S_v)$, $\nu_v$ and $\sigma_v$.
\end{theorem}
Propositions \ref{prop1}, \ref{prop2}, \ref{prop3}, \ref{prop5} and \ref{prop4} will imply the proof.

%%%%%%%%%%%%%%%%%%%%%%%%%%%%%%%%%%%%%%%%%%%%%%%%%%%%%%%%%%%%%%%%%%%%%%%%%%%%%%%%%%%%%%%%%%%%%%%%%%%%%%%%%%%%%%%%%%
\section{Examples}

\begin{example}
Let $p=ABCD|ECFA|DFBE$, where $|A|=|E|=b+$, $|B|=b-$, $|C|=|F|=a-$ and $|D|=a+$. 
This corresponds to the Borromean rings illustrated in Fig.\ \ref{fig.e1}.
Then since $w_1^\varepsilon=AC^{-1}$, 
\begin{align*}
\rho^2(w_1^\varepsilon) &=AC^{-1} \\
\rho^3(w_1^\varepsilon) &=FE^{-1}AEF^{-1}E^{-1}C^{-1}E. 
\end{align*}
Thus we have
\begin{align*}
\eta (\rho^2(w_1^\varepsilon)) &=a_2a_2^{-1} \\
\eta(\rho^3(w_1^\varepsilon)) &=a_3a_3^{-1}a_2a_3a_3^{-1}a_3^{-1}a_2^{-1}a_3. 
\end{align*}
Therefore $\mu(2,1)=\mu(3,1)=0$ and $\mu(2,3,1)=-1$.  
Similarly $\mu(1,2)=\mu(1,3)=\mu(2,3)=\mu(3,2)=0$ and so 
$\Delta(2,3,1)=0$.
Hence $\bar{\mu}(2,3,1) \equiv -1 \pmod0$.

\begin{figure}[H]
\begin{center}
\includegraphics[scale=0.2]{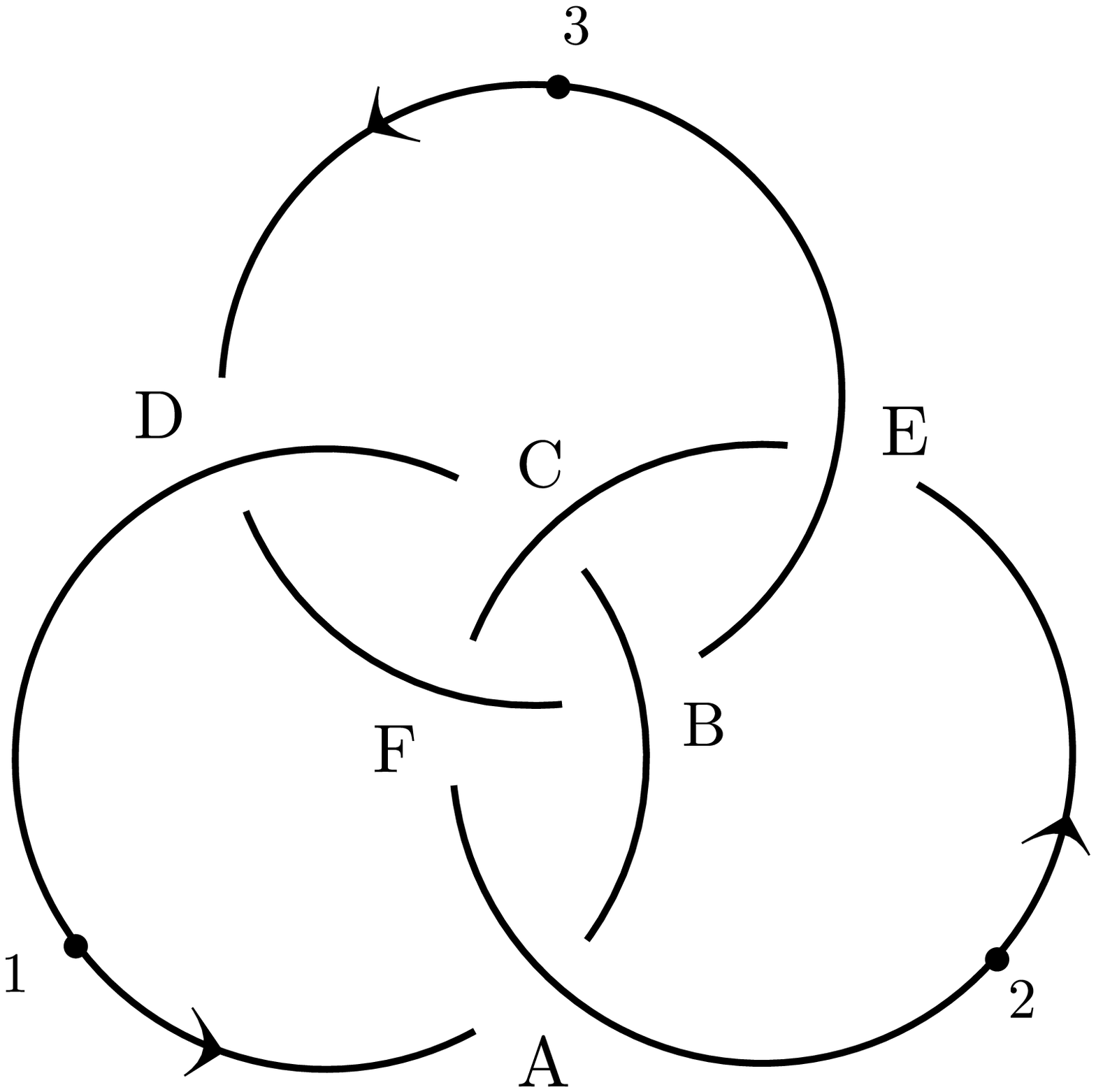}
\caption{}
\label{fig.e1}
\end{center}
\end{figure}
\end{example}

\begin{example}
Let $p=A_1B_1A_2B_2 \ldots A_nB_n|A_1B_1A_2B_2 \ldots A_nB_n$, where $|A_j|=b+$ and $|B_j|=a+$.
This corresponds to the $(2n, 2)$ torus link illustrated in Fig.\ \ref{fig.2}.
Then since $w_1^\varepsilon=A_1 \ldots A_n$ and $w_2^\varepsilon=B_1 \ldots B_n$,
\begin{align*}
\eta (\rho^2(w_1^\varepsilon)) &=a_2 \ldots a_2 \\
\eta(\rho^2(w_2^\varepsilon)) &=a_1 \ldots a_1. 
\end{align*}
Therefore $\mu(2,1)=n$ and $\mu(1,2)=n$, and so $\bar{\mu}(2,1)=n$ and $\bar{\mu}(1,2)=n$.

\begin{figure}[H]
\begin{center}
\includegraphics[scale=0.4, angle=0]{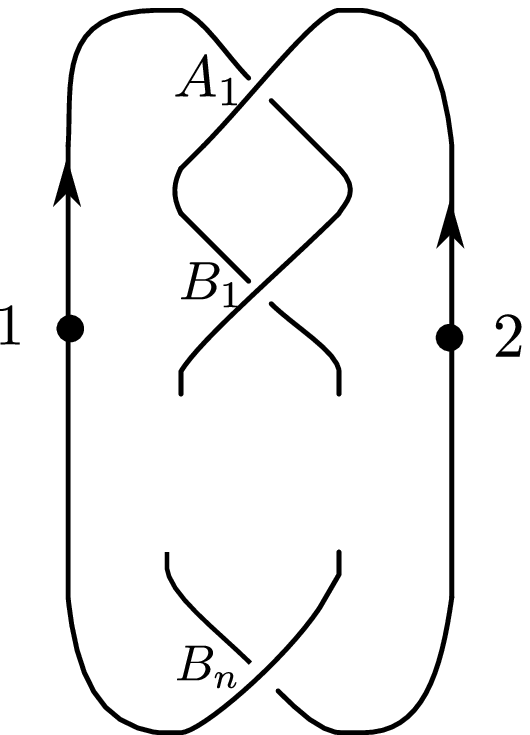}
\caption{}
\label{fig.2}
\end{center}
\end{figure}
\end{example}

\begin{example}
Let $p=A_1A_2 \ldots A_n|A_1A_2 \ldots A_n$, where $|A_j|=b+$.
This corresponds to the virtual link illustrated in Fig.\ \ref{fig.e2}.
Then since $w_1^\varepsilon=A_1 \ldots A_n$ and $w_2^\varepsilon=\emptyset$,
\begin{align*}
\eta (\rho^2(w_1^\varepsilon)) &=a_2 \ldots a_2 \\
\eta(\rho^2(w_2^\varepsilon)) &=\emptyset. 
\end{align*}
Therefore $\mu(2,1)=n$ and $\mu(1,2)=0$, and so $\bar{\mu}(2,1)=n$ and $\bar{\mu}(1,2)=0$.

\begin{figure}[H]
\begin{center}
\includegraphics[scale=0.5, angle=0]{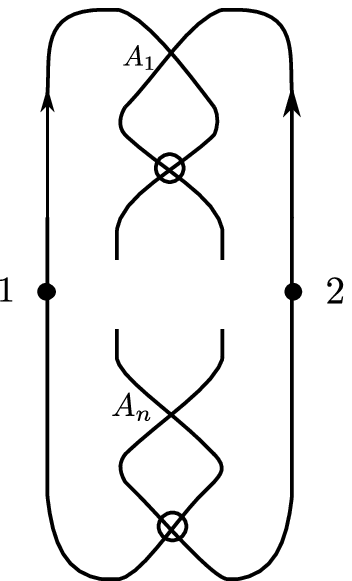}
\caption{}
\label{fig.e2}
\end{center}
\end{figure}
\end{example}

\begin{example}\label{ce}
Let $p=AB|FBCDAE|DC|EF$, where $|A|=|B|=|E|=b+$, $|C|=b-$, $|D|=a-$ and $|F|=a+$. 
This corresponds to the link illustrated in Fig.\ \ref{fig.7}.
Then since $w_1^\varepsilon=AB$, 
\begin{align*}
\rho^2(w_1^\varepsilon) &=AB \\
\rho^3(w_1^\varepsilon) &=DAD^{-1}B. 
\end{align*}
Thus we have
\begin{align*}
\eta (\rho^2(w_1^\varepsilon)) &=a_2a_2 \\
\eta(\rho^3(w_1^\varepsilon)) &=a_3a_2a_3^{-1}a_2. 
\end{align*}
Therefore $\mu(2,1)=2$, $\mu(3,1)=0$, $\mu(2,3,1)=-1$ and $\mu(3,2,1)=1$.  
Similarly $w_2^\varepsilon=D^{-1}E$ and $w_3^\varepsilon=C^{-1}$.
Thus we have 
\begin{align*}
\eta (\rho^2(w_2^\varepsilon)) &=a_3^{-1}a_4 \\
\eta(\rho^2(w_3^\varepsilon)) &=a_2^{-1}. 
\end{align*}
Therefore $\mu(1,2)=0$, $\mu(3,2)=-1$, $\mu(1,3)=0$ and $\mu(2,3)=-1$.
So we have $\Delta(2,3,1)=\Delta(3,2,1)=1$.
As a result, $\bar{\mu}(2,3,1) \equiv \bar{\mu}(3,2,1) \equiv 0 \pmod1$.

\begin{figure}[H]
\begin{center}
\includegraphics[scale=0.6, angle=0]{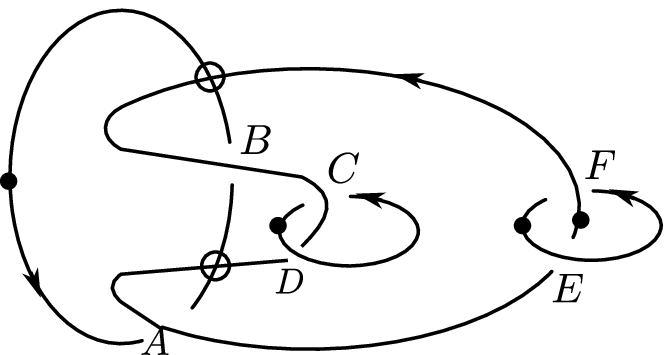}
\caption{}
\label{fig.7}
\end{center}
\end{figure}
\end{example}

\begin{example} \label{counter.ex}
Let $p=ABCD|DAEFBC|EF|\emptyset$, where $|A|=|C|=|F|=a+$, $|B|=|D|=|E|=b+$.
This corresponds to the link illustrated in Fig.\ \ref{fig.e3}.
Then since $w_1^\varepsilon=BD$, 
\begin{align*}
\rho^2(w_1^\varepsilon) &=BD \\
\rho^3(w_1^\varepsilon) &=E^{-1}A^{-1}BAED. 
\end{align*}
Thus we have
\begin{align*}
\eta (\rho^2(w_1^\varepsilon)) &=a_2a_2 \\
\eta(\rho^3(w_1^\varepsilon)) &=a_3^{-1}a_1^{-1}a_2a_1a_3a_2. 
\end{align*}
Therefore $\mu(2,1)=2$, $\mu(3,1)=0$ and $\mu(2,3,1)=1$.  
Similarly $w_2^\varepsilon=AEC$ and $w_3^\varepsilon=F$.
Thus we have 
\begin{align*}
\eta (\rho^2(w_2^\varepsilon)) &=a_1a_3a_1 \\
\eta(\rho^2(w_3^\varepsilon)) &=a_2. 
\end{align*}
Therefore $\mu(1,2)=2$, $\mu(3,2)=1$, $\mu(1,3)=0$ and $\mu(2,3)=1$.
So we have $\Delta(2,3,1)=1$.
As a result, $\bar{\mu}(2,3,1) \equiv 0 \pmod1$.

\begin{figure}[H]
\begin{center}
\includegraphics[scale=0.6, angle=0]{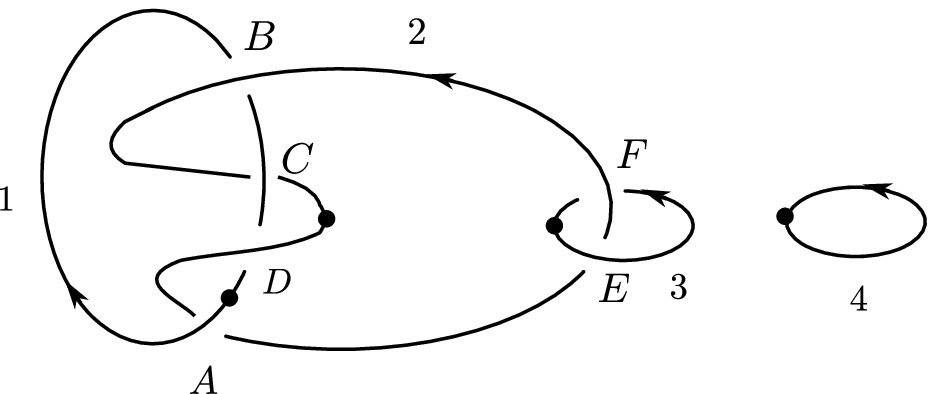}
\caption{}
\label{fig.e3}
\end{center}
\end{figure}
\end{example}

\begin{remark}\label{counter}
Our definition is different from that of Dye and Kauffman.  
The Dye-Kauffman's definition should have included "up to cyclic permutations" in their definition of  $ \Delta $.  
This would eliminate the dependence on base points.  
(Change a base point for Example \ref{counter.ex})
\end{remark}

\begin{example}
We consider two Gauss diagrams in Fig.\ \ref{fig.e6}, which represent virtual string links, respectively.
The closures of these virtual string links both coincide with the virtual link in Fig.\ \ref{fig.7}.
Using Kravchenko-Polyak's definition, we compute $\mu$ and compare it with our definition.
\begin{align*}
{\mu}(2, 3, 1)(G_1) &=Z_{23,1}(G_1) \\
&= \left\langle  \includegraphics[scale=0.3, angle=0]{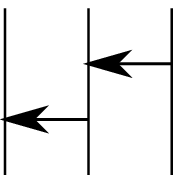} + \includegraphics[scale=0.3, angle=0]{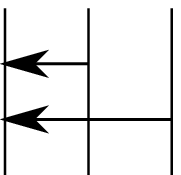} - \includegraphics[scale=0.3, angle=0]{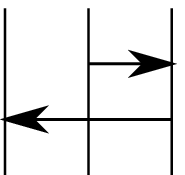}, G_1\right\rangle \\
&=-1 + 0 + 0 = -1.
\end{align*}
On the other hand, 
\begin{align*}
{\mu}(2, 3, 1)(G_2) &= Z_{23,1}(G_2) \\
&= \left\langle  \includegraphics[scale=0.3, angle=0]{arrowdia1.eps} + \includegraphics[scale=0.3, angle=0]{arrowdia2.eps} - \includegraphics[scale=0.3, angle=0]{arrowdia3.eps}, G_2\right\rangle \\
& =0 + 0 + 0 = 0.
\end{align*}
Modulo lower degree invariants of Kravchenko-Polyak's definition, that is the greatest common divisor of ${\mu}(2, 3)$, ${\mu}(3, 2)$, ${\mu}(2, 1)$, ${\mu}(1, 2)$, ${\mu}(3, 1)$ and ${\mu}(1, 3)$, both ${\mu}(2, 3, 1)(G_1)$ and  ${\mu}(2, 3, 1)(G_2)$ coincide with $\bar{\mu}(2, 3, 1)$ of the virtual link shown in Fig.\ \ref{fig.7}.
\begin{figure}[H]
\begin{center}
\includegraphics[scale=0.4, angle=0]{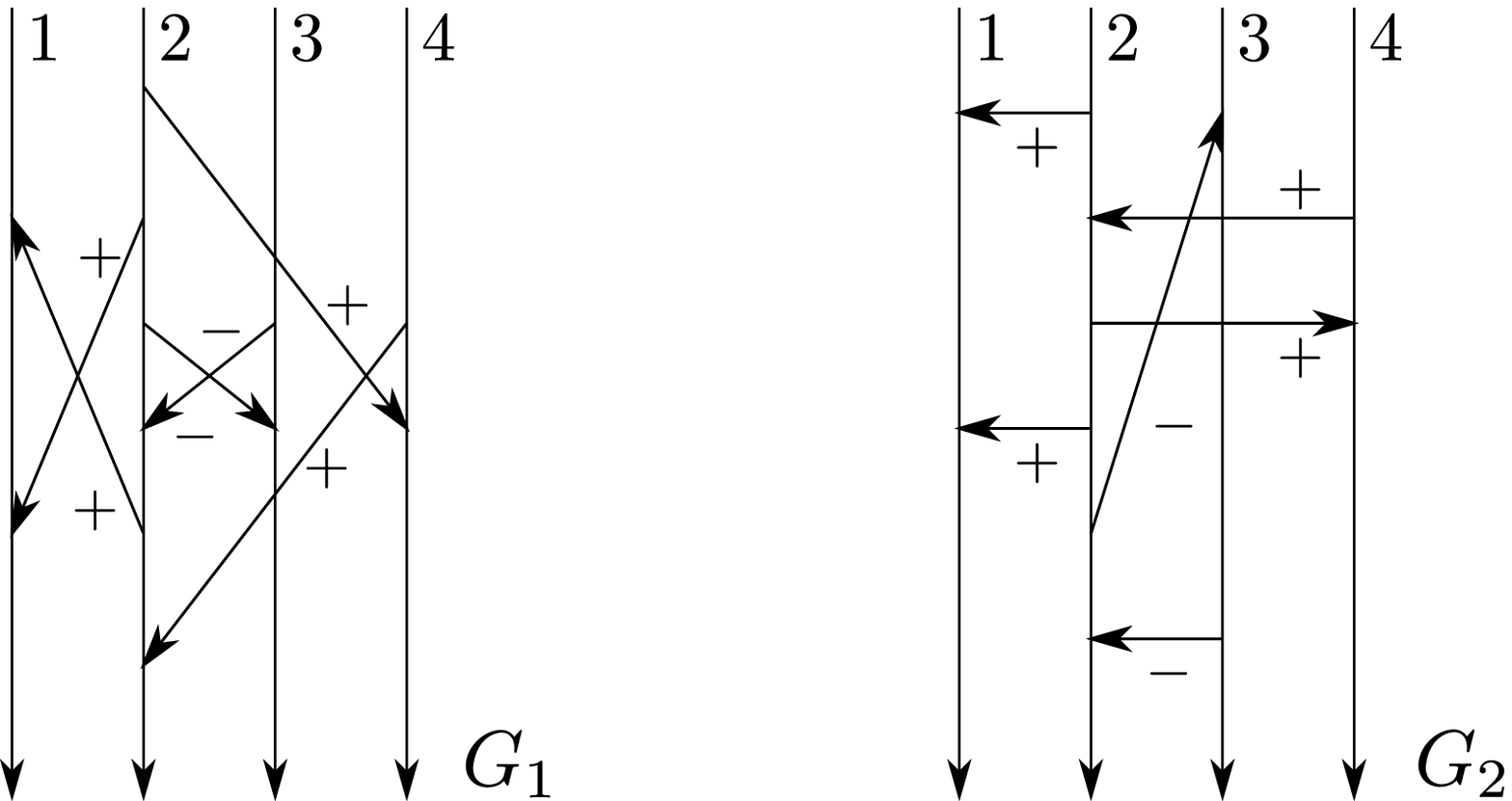}
\caption{}
\label{fig.e6}
\end{center}
\end{figure}
\end{example}

%%%%%%%%%%%%%%%%%%%%%%%%%%%%%%%%%%%%%%%%%%%%%%%%%%%%%%%%%%%%%%%%%%%%%%%%%%%%%%%%%%%%%%%%%%%%%%%%%%%%%%%%%%%%%%%%%%%%%%%%%%%%%%%%%%%%%%%%%%%
\section{Preliminaries to the proof of invariance}
It is sufficient to prove that for any two nanophrases $p$ and $p'$ related by either isomorphisms, H1 moves, H2 moves, H3 moves,
shift moves, or self crossing moves, and for each sequence $c_1, c_2, \ldots ,c_u,i$ of distinct integers between $1$ and $n$, 
$\bar{ \mu} (p; c_1, c_2, \ldots ,c_u,i)$ is equal to $\bar{ \mu}(p'; c_1, c_2, \ldots ,c_u,i)$. 
It is easy to see that this holds for an isomorphism.
To show invariance under each move we use graphs. 

Given a nanophrase  $p = w_1|w_2| \cdots |w_n$, 
we defined its signed $i$th component  
$w_i^{\varepsilon} = A_{i1}^{\varepsilon_{i1}}A_{i2}^{\varepsilon_{i2}} 
\cdots A_{im_i}^{\varepsilon_{im_i}}$  in the previous section.  

For each $i$, we will construct rooted trees  $T_{ij}^q(p) = T_{A_{ij}}^q(p)$ 
which correspond to expanding words  $\rho^q(A_{ij}^{\varepsilon_{ij}})$'s of  $A_{ij}^{\varepsilon_{ij}}$'s, 
and a forest $F_i^q=F_i^q(p)$ which is a sequence of rooted trees, by assembling  $T_{ij}^q$'s.
To do this, 
recall how we defined  $\rho^q(A_{ij}^{\varepsilon_{ij}})$  from   
$\rho^{q-1}(A_{ij}^{\varepsilon_{ij}})$.  
Since the operation of  $\rho$  on 
$\rho^{q-1}(A_{ij}^{\varepsilon_{ij}})$ 
is to insert several words in 
$\rho^{q-1}(A_{ij}^{\varepsilon_{ij}})$,   
$\rho^{q-1}(A_{ij}^{\varepsilon_{ij}})$  is obtained from 
$\rho^{q}(A_{ij}^{\varepsilon_{ij}})$  by deleting some subwords.  
If we denote this inclusive relation by  $<$,  
we have a natural inclusion of words, 
\begin{equation*} 
	\rho^{2}(A_{ij}^{\varepsilon_{ij}}) <  
	\rho^{3}(A_{ij}^{\varepsilon_{ij}}) < \cdots <    
	\rho^{q}(A_{ij}^{\varepsilon_{ij}}).  
\end{equation*} 
Then, 
we assign two natural numbers to each letter in  
the word  $\rho^{q}(A_{ij}^{\varepsilon_{ij}})$.  
One is the depth.  
To each letter in  $\rho^{q}(A_{ij}^{\varepsilon_{ij}})$, 
we assign the depth  $d$  if it appears in 
$\rho^{d+2}(A_{ij}^{\varepsilon_{ij}})$  but not in 
$\rho^{d+1}(A_{ij}^{\varepsilon_{ij}})$.   
The depth is an intrinsic invariant for a letter  
in  $\rho^{q}(A_{ij}^{\varepsilon_{ij}})$  independent from  $q$.  
The other is its location in  $\rho^{q}(A_{ij}^{\varepsilon_{ij}})$.  
We assign the location  $g$  to the  $g$th letter in  $\rho^{q}(A_{ij}^{\varepsilon_{ij}})$.  
Note that the location  $g$  depends on  $q$.  
We now define an increasing sequence of rooted trees, 
\begin{equation}\label{Eq:RootedTree}
	T_{ij}^2 \subset  
	T_{ij}^3 \subset \cdots \subset     
	T_{ij}^q  \subset \cdots .
\end{equation} 
The vertices of  $T_{ij}^q$  consists of 
letters in  $\rho^{q}(A_{ij}^{\varepsilon_{ij}})$.  
To each vertex, 
we assign its depth. 
We join two vertices  $v$  and  $w$  by a directed edge 
from  $v$  to  $w$  
if the depth  $d$  of  $v$  is equal to the depth of  $w$ minus  $1$  and 
moreover  $w$  is a letter that appears in the image of a signed letter corresponding to $v$ under $\rho^{3}$  but 
not under  $\rho^{2}$.  
Thus we get a sequence of rooted trees (\ref{Eq:RootedTree})  
based on the letter  $A_{ij}^{\varepsilon_{ij}}$  in  $w_i^{\varepsilon}$.  
The root of each tree is represented 
uniquely by  $A_{ij}^{\varepsilon_{ij}}$ of depth 0.

For further discussion, 
we assign a label to each vertex in  $T_{ij}^q$.  
It is a quadruple consisting of 
the location  $g$  of the letter in $\rho^{q}(A_{ij}^{\varepsilon_{ij}})$,  
its letter in  $\mathcal{A}$  without sign, 
the index  $k$  of its image under  $\eta$,
and the sign  $\varepsilon$ of its letter.
If we regard $T_{ij}^q$  as a labeled rooted tree,  
then, 
the vertices can be distinguished by the first label.   
On the other hand, 
the inclusive relation in  (\ref{Eq:RootedTree})  still does make 
sense if we ignore the first label.   

We now denote the forest of labeled rooted trees  $T_{ij}^q(p)$'s  
by  $F_i^q(p)$,  
where  $j$  ranges between  $1$  and  $m_i$.  
Note that  $F_i^q(p)$  contains  $m_i$  roots of depth $0$.

\begin{example}\label{ex2.555}
Consider the nanophrase $p=AB|CDB|DEA|FFCE $, where $|A|=|B|=|C|=b_+,|D|=b_-,|E|=|F|=a_-$ (Example $\ref{ex2.55}$).
Then we have
\begin{align*}
&\rho^2 (w_1^\varepsilon ) = AB, ~~\rho^3 (w_1^\varepsilon ) = EDAD^{-1}E^{-1}C^{-1}BC,  \\
&\rho^4 (w_1^\varepsilon ) = EC^{-1}DCAC^{-1}D^{-1}CE^{-1}C^{-1}BC, \\
&\eta (\rho^4(w_1^\varepsilon)) = a_4a_4^{-1}a_2a_4a_3a_4^{-1}a_2^{-1}a_4a_4^{-1}a_4^{-1}a_2a_4.
\end{align*}
Then $F_1^4(p)$ is the following forest.
\begin{figure}[H]
\begin{center}
\includegraphics[scale=0.7]{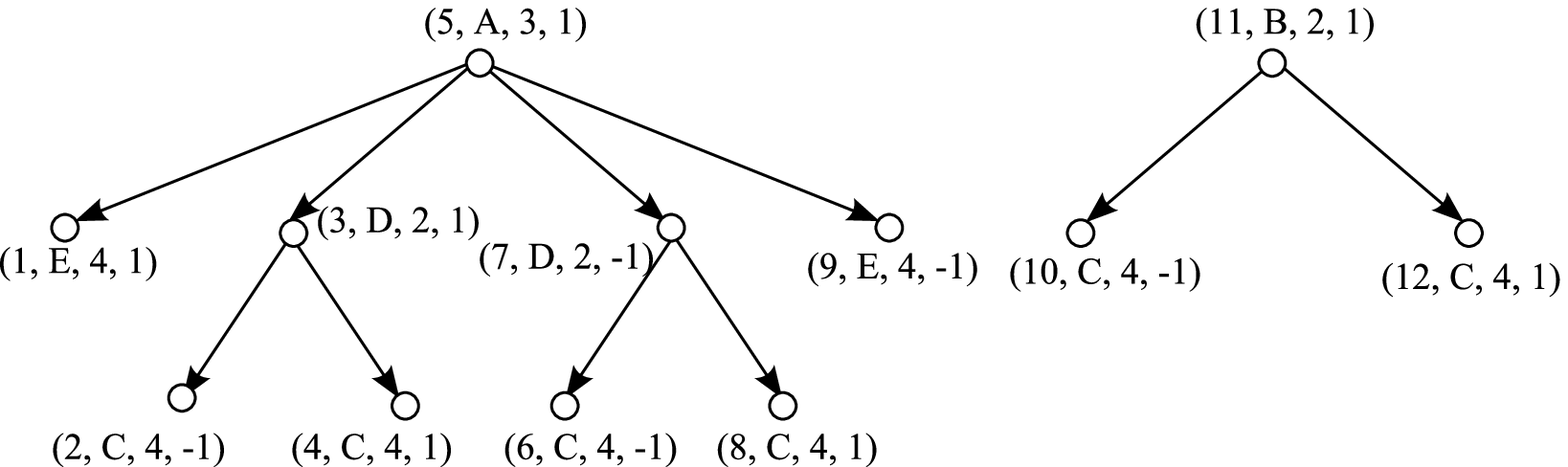}
\end{center}
\end{figure}
\end{example}

Let $F$ be a forest of rooted trees with labels in $\mathbb{N} \times \mathcal{A} \times \{1,2, \cdots ,n \} \times \{\pm1 \}$ 
such that vertices can be distinguished by the first label.  
Then, for any sequence  $c_1, \, c_2, \cdots, c_u$  of pairwise different integers between 
$1$  and  $n$, 
let  $\mathcal{S}(F, c_1, \cdots, c_u)$  be the set of 
subforests of  $F$  satisfying the following conditions: 
\begin{enumerate}
\item
	Each member of  $\mathcal{S}(F, c_1, c_2, \cdots, c_u$)  has  $u$  vertices 
with $c_1, \, c_2, \cdots, c_u$, respectively, as the third label.  
\item 
Let $d_j$ denote the distinguishing the first label of the vertex with $c_j$  as the third label.
Then $d_1 < d_2 < \cdots < d_u$.
\item
	Each connected component contains a root.  
\end{enumerate} 
In the rest of the paper, we call $F$ a {\it labeled forest} instead of a forest of rooted trees with labels.

We then define sets  $\mathcal{S}_e(F, c_1, \cdots, c_u)$  and  $\mathcal{S}_o(F, c_1, \cdots, c_u)$  as follows.  
For any sequence $c_1, c_2, \cdots, c_u$ of pairwise different
integers  between  $1$  and  $n$,  
let  $\mathcal{S}_e(F, c_1, \cdots, c_u)$  and  $\mathcal{S}_o(F, c_1, \cdots, c_u)$  
be subsets of  $\mathcal{S}(F, c_1, \cdots, c_u)$  which 
consists of subforests for which the cardinality of 
vertices with  $\varepsilon_j = -1$  is even or odd respectively.  
Note that  $\varepsilon_j$  is the sign of the fourth label.  
Also note that 
we have the identity 
\begin{equation*}
	\mathcal{S}(F, c_1, \cdots, c_u) = 
	\mathcal{S}_o(F, c_1, \cdots, c_u) \sqcup
	\mathcal{S}_e(F, c_1, \cdots, c_u). 
\end{equation*}

\begin{example}\label{ex2.5555}
Consider the nanophrase $p=AB|CDB|DEA|FFCE $, where $|A|=|B|=|C|=b_+,|D|=b_-,|E|=|F|=a_-$ (Example $\ref{ex2.555}$).
Then we have the following, where we note that we write only the first label. 
\begin{figure}[H]
\begin{center}
\includegraphics[scale=0.3]{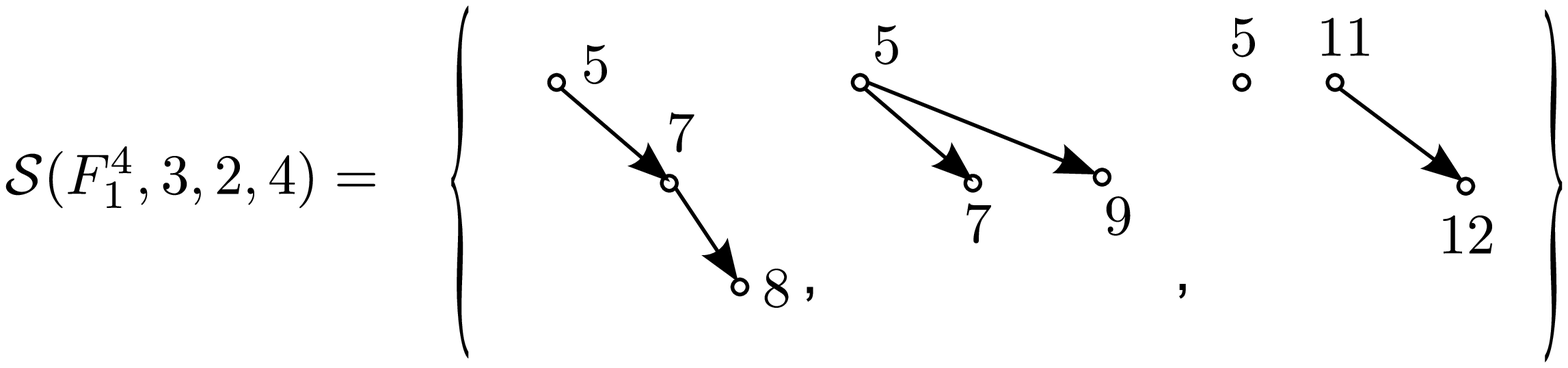}
\end{center}
\end{figure}

\begin{figure}[H]
\begin{center}
\includegraphics[scale=0.3]{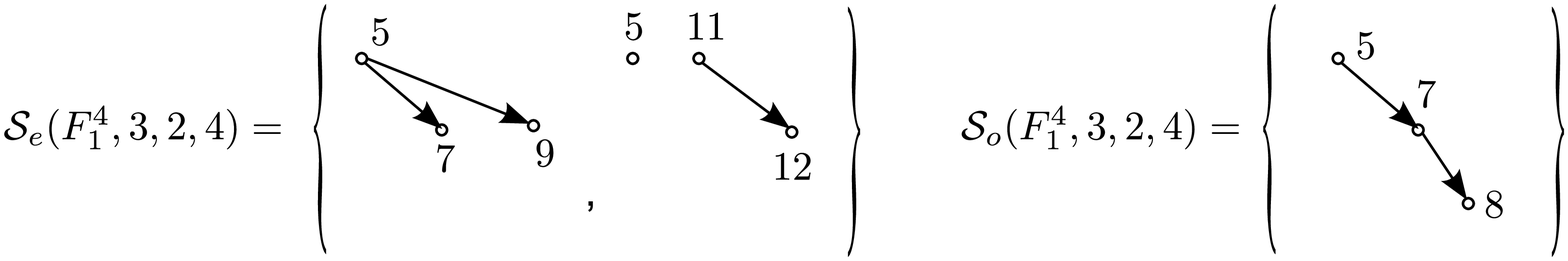}
\end{center}
\end{figure}
\end{example}

Let $E$ denote the set of all power series $\sum \lambda  (e_1, e_2, \ldots ,e_v) \kappa _{e_1} \kappa _{e_2}  \ldots \kappa _{e_v}$ 
in $\mathbb{Z} [[\kappa_1,\kappa_2, \ldots, \kappa_n]]$ such that 
at least two indices in $\{ e_1, e_2, \ldots ,e_v \}$ coincide.
It is easy to see that $E$ is a two-sided ideal.

\begin{lemma}\label{lem0} 
If the indices $c_1, c_2, \ldots ,c_u$ are pairwise distinct, then for any $i$ and $q$ the coefficient of $\kappa _{c_1} \cdots \kappa _{c_u}$ in $\varphi \circ \eta (\rho^q (w_i^\varepsilon))$
is equal to the cardinality of $\mathcal{S}_{\text{e}}({F}_i^q,c_1, \ldots,c_{u})$ minus the cardinality of $\mathcal{S}_{\text{o}}({F}_i^q,c_1, \ldots,c_{u})$. 
\end{lemma} 
\begin{proof}
We consider the coefficient of $\kappa _{c_1} \cdots \kappa _{c_u}$ in $\varphi \circ \eta (\rho^q (w_i^\varepsilon))$.
$\varphi \circ \eta ( \rho^q (w_i^\varepsilon ) )$ is the product of some elements either of the form $(1+ \kappa _{c})$ or $(1 - {\kappa _{d}} + {\kappa _{d}}^2 - {\kappa _{d}}^3 + \cdots $) in $\mathbb{Z} [[\kappa_1,\kappa_2, \ldots, \kappa_n]]$, where $1 \leq c,d \leq n$.
Since the indices $c_1, c_2, \ldots ,c_u$ are pairwise distinct by assumption,
it is sufficient to consider the coefficient of $\kappa _{c_1} \cdots \kappa _{c_u}$ modulo $E$.
Therefore we regard $1 - {\kappa _{}} + {\kappa _{}}^2 - {\kappa _{}}^3 + \cdots $ as $1 - {\kappa _{}}$ modulo $E$.
Let $\eta (A_{ij}) = a_{e_j}$, then 
\[ \varphi \circ \eta (A_{ij}^{\varepsilon_{ij}}) \equiv 1+ \varepsilon_{ij} \kappa_{e_j} \mod{E}. \]
On the other hand, $\rho^q (w_i^\varepsilon )$ is the product of $\rho^q (A_{ij}^{\varepsilon_{ij}} )$'s ($1 \leq j \leq m_i$).
Moreover by the definition of $\rho$, 
\[ \rho^q (A_{ij}^{\varepsilon_{ij}} ) = \rho^{q-1}(x_{ij}^{-1})A_{ij}^{\varepsilon_{ij}} \rho^{q-1}(x_{ij}). \]
Then the only letter with depth 0 in $\rho^q (A_{ij}^{\varepsilon_{ij}} )$ is $A_{ij}^{\varepsilon_{ij}}$ in the right hand side of this identity.
We deduce that 
\begin{align}
\varphi \circ \eta (\rho^q (A_{ij}^{\varepsilon_{ij}})) \equiv \varphi \circ \eta (\rho^{q-1}(x_{ij}^{-1})) \cdot (1+ \varepsilon_{ij} \kappa_{e_j}) \cdot \varphi \circ \eta (\rho^{q-1}(x_{ij})) 
\mod{E},  \label{expand} 
\end{align}
where $e_j$ is some integer between 1 and $n$. 
First of all, we expand the term corresponding to the letter with depth 0, that is $(1+ \varepsilon_{ij}{\kappa} _{e_j})$ in (\ref{expand}).
Then we have 
\begin{align}
\varphi \circ \eta (\rho^{q-1}(x_{ij}^{-1})) \cdot 1 \cdot \varphi \circ \eta (\rho^{q-1}(x_{ij})) + \varphi \circ \eta (\rho^{q-1}(x_{ij}^{-1})) \cdot \varepsilon_{ij} \kappa_{e_j} \cdot \varphi \circ \eta (\rho^{q-1}(x_{ij})), \label{expand2} 
\end{align}
Secondly, we cancel out the first term in (\ref{expand2}).
Then we have 
\begin{align*}
1 + \varphi \circ \eta (\rho^{q-1}(x_{ij}^{-1})) \cdot \varepsilon_{ij} \kappa_{e_j} \cdot \varphi \circ \eta (\rho^{q-1}(x_{ij})). 
\end{align*}
We repeat these two steps according to depths, inductively.
Then it follows from how to expand the expression and how to construct the tree $T_{ij}^q$ 
that there exists a natural bijection between a term in this expanded expansion and a connected subtree of $T_{ij}^q$ which contains the root.
The coefficient of each term in this expanded expression is $\pm 1$.
On the other hand, for each term in this expanded expansion, we consider vertices of the subtree corresponding to it.
The number of vertices with $-1$ as the fourth label is even (resp. odd) if and only if the coefficient of the term in this expansion is 1 (resp. $-1$). 
Similar considerations apply to a relation between $\varphi \circ \eta (\rho^q (w_i^\varepsilon))$ and a forest $F_i^q$.  
Therefore 
\begin{align*}
\text{ the coefficient of } \kappa _{c_1} \cdots \kappa _{c_u} & \text{ in } \varphi \circ \eta (\rho^q (w_i^\varepsilon )) \\
&= \# \mathcal{S}_{\text{e}}({F}_i^q,c_1, \ldots ,c_{u}) - \# \mathcal{S}_{\text{o}}({F}_i^q,c_1, \ldots ,c_{u}).
\end{align*}
\end{proof}

% \begin{remark}
% We take subwords $a_{c_1}a_{c_2} \cdots a_{c_u}$ of the word $\eta (\rho^q (w_i^\varepsilon )) $, then for each subword, using the signs of the each $a_{c_j}$ of the subword, we can determine whether the term $\kappa _{c_1} \cdots \kappa _{c_u}$ associated with the subword is added or subtracted.  
% \end{remark}

\begin{example}\label{2.55555}
Consider the nanophrase $p=AB|CDB|DEA|FFCE $, where $|A|=|B|=|C|=b_+,|D|=|E|=b_-,|F|=a_-$ (Example $\ref{ex2.55}$). 
Then recall from Example $\ref{ex2.5}$ that we have 
\[ \rho^3(w_3^\varepsilon) = {C}^{-1}{D}^{-1}CE^{-1}, \]
where the depths of $D^{-1}$ and $E^{-1}$ are 0, and the depths of $C$ and $C^{-1}$ are 1. 
Thus,
\begin{align*}
\eta (\rho^3(w_3^\varepsilon)) &= {a}_4^{-1}{a}_2^{-1}a_4{a}_4^{-1},   \\
\varphi \circ \eta (\rho^3(w_3^\varepsilon)) &=(1- {\kappa _4} + \kappa _4^2 - \cdots)(1- \kappa _2 + \kappa _2^2 - \cdots)(1+\kappa _4)(1- \kappa _4 + \kappa _4^2 - \cdots).  
\end{align*}
Therefore 
\begin{align}
\varphi \circ & \eta (\rho^3(w_3^\varepsilon))  \notag \\
&  \equiv (1- {\kappa _4})(1- \kappa _2 )(1+\kappa _4)(1- \kappa _4) \mod{E} \label{l.e1}  \\
\begin{split}
& = (1- {\kappa} _4)1(1+\kappa _4)1 + (1- {\kappa} _4)1(1+\kappa _4)(- {\kappa} _4)  \\
& \hspace{3em} + (1- {\kappa} _4)(- {\kappa} _2)(1+\kappa _4)1 + (1- {\kappa} _4)(- {\kappa} _2)(1+\kappa _4)(- {\kappa} _4) \label{l.e2} 
\end{split} \\
& \equiv 1 -{\kappa} _4 + (1- {\kappa} _4)(- {\kappa} _2)(1+\kappa _4) + (1- {\kappa} _4)(- {\kappa} _2)(1+\kappa _4)(- {\kappa} _4)  \mod{E} \label{l.e3} \\
\begin{split}
&= 1 -{\kappa} _4  + \Bigl( - {\kappa} _2 - {\kappa} _2\kappa _4 + {\kappa} _4 {\kappa} _2 +  {\kappa} _4 {\kappa} _2\kappa _4 \Bigr) \\
& \hspace{9em}+  \Bigl( {\kappa}_2 {\kappa} _4 + {\kappa} _2\kappa _4 {\kappa} _4 - {\kappa}_4{\kappa} _2{\kappa} _4 - {\kappa} _4{\kappa} _2\kappa _4{\kappa} _4 \Bigr).  \label{l.e4} 
\end{split}  
\end{align}
Here for the right hand term of (\ref{l.e1}), we expand the terms corresponding to letters with depth 0, that is the second and fourth terms.
Moreover for the right hand term of (\ref{l.e2}), we simplify the expansion by canceling them out.
In other words, we change $(1-\kappa _4)1(1+\kappa _4)$ into 1.
% that is we cancel $(1-\kappa _4)1(1+\kappa _4)$, we obtain (\ref{l.e3}).
In addition, for the right hand term of (\ref{l.e3}), we expand the terms corresponding to letters with depth 1, that is $(1-\kappa _4)$ and $(1+\kappa _4)$. 

On the other hand, we have
\begin{figure}[H]
\begin{center}
\includegraphics[scale=0.5]{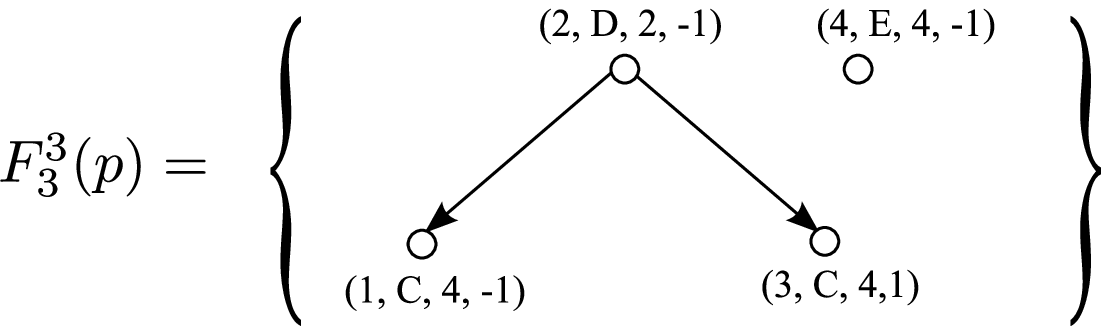}
\end{center}
\end{figure}

The right hand term of (\ref{l.e4}) corresponds to the subforests as follows, where we note that we write only the first label.
\begin{figure}[H]
\begin{center}
\includegraphics[scale=0.3]{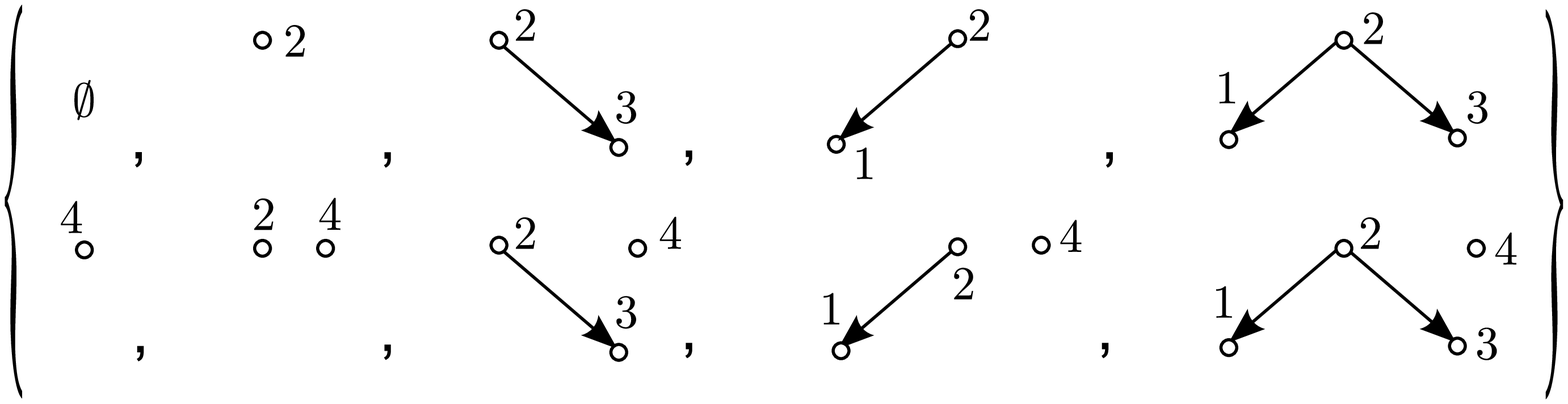}
\end{center}
\end{figure}
\end{example}
Then, for example, the sum of the coefficient of $\kappa_2 \kappa_4$ in (\ref{l.e4}) is 0.
On the other hand, 
$\mathcal{S}_{\text{e}}({F}_3^3, 2,4)$ is the set which consists of a single forest, which has only two roots labeled 2 and 4 and no descendant (second term of the second row in the above set).
$\mathcal{S}_{\text{o}}({F}_3^3, 2,4)$ is the set which consists of a single forest, which has only one root labeled 2, with one descendant labeled 3 (third term of the
first row in the above set).
Therefore we have 
\[ \# \mathcal{S}_{\text{e}}({F}_3^3, 2,4) - \# \mathcal{S}_{\text{o}}({F}_3^3, 2,4) =1-1 = 0. \] 

%%%%%%%%%%%%%%%%%%%%%%%%%%%%%%%%%%%%%%%%%%%%%%%%%%%%%%%%%%%%%%%%%%%%%%%%%%%%%%%%%%%%%%%%%%%%%%%%%%%%%%%%%%%%%%%%%%%%%%%%%%%%%%%%%%%%%%%%%%%%%
\section{Proof of invariance for H1, H2, H3 and self crossing moves}

\begin{proposition}\label{prop1}
If the indices $c_1, c_2, \ldots ,c_u,i$ are a sequence of integers between 1 and $n$, then $\mu (p;c_1, \ldots ,c_{u},i)=\mu (p';c_1, \ldots ,c_{u},i)$ if $p$ and $p'$ are related by an H1 move. 
\end{proposition}
\begin{proof}
Let $p'$ be $w'_1|w'_2| \cdots |w'_n$.
Let us $p=xAAy$ and $p'=xy$.
The letter $A$ appears twice in a certain component.
It follows from the definition of $\rho^q (w_i^\varepsilon )$ that the letter $A$ is not contained in $\rho^q (w_i^\varepsilon )$ and thus does not affect $\rho^q (w_i^\varepsilon )$.
Therefore $\rho^q (w_i^\varepsilon )$ is equal to $\rho^q ({w'_i}^\varepsilon )$, and so for any sequence $c_1, c_2, \ldots ,c_u,i$ $\mu (p; c_1, c_2, \ldots ,c_u,i)$ is equal to $\mu (p'; c_1, c_2, \ldots ,c_u,i)$. 
\end{proof}

\begin{proposition}\label{prop2}
If the indices $c_1, c_2, \ldots ,c_u,i$ are a sequence of integers between 1 and $n$, then $\mu (p;c_1, \ldots ,c_{u},i)=\mu (p';c_1, \ldots ,c_{u},i)$ if $p$ and $p'$ are related by an H2 move. 
\end{proposition}
\begin{proof}
Let us $p=xAByBAz$ and $p'=xyz$, 
where $|A|$ is equal to $\tau_v (|B|)$.
We need to consider the following two cases.
\begin{description}
\item[{\rm Case 1}] The subwords $AB$ and $BA$ appear in a certain component. \\
(i.e. $p$ is a nanophrase of the form  $... |... AB ... BA ...| ...$ ) 
\item[{\rm Case 2}] The subwords $AB$ and $BA$ appear in different components. \\
(i.e. $p$ is a nanophrase of the form  $... |... AB ...|...|... BA ...|...$ ) 
\end{description}  

Case 1: 
The letters $A$ and $B$ appear twice in a certain component.
It follows from the definition of $\rho^q (w_i^\varepsilon )$ that the letters $A$ and $B$ are contained in $\rho^q (w_i^\varepsilon )$ and thus do not affect $\rho^q (w_i^\varepsilon )$.
Therefore $\rho^q (w_i^\varepsilon )$ is equal to $\rho^q ({w'_i}^\varepsilon )$, and so $\mu (p;c_1, \ldots ,c_{u},i)$ is equal to $\mu (p';c_1, \ldots ,c_{u},i)$. 

Case 2: 
Let $h$ be the order of the component in which the subword $AB$ occurs 
and $k$ the order of the component in which the subword $BA$ occurs. 
We show that $\varphi \circ \eta ( \rho^q (w_i^\varepsilon ))$ is equal to $\varphi \circ \eta ( \rho^q ({w'_i}^\varepsilon ))$.
Let us $|A|=b_+$ and $|B|=a_-$.
(We can show the other cases in the same way.)
Then the sign of $A$ in the $h$th component is 1 and the sign of $B$ in the $h$th component $-1$.
Moreover, the signs of $A$ and $B$ in the $k$th component are 0, and so they do not appear directly in $\rho^q ({w}_i^\varepsilon )$. 
Since $A$ and $B$ in the $h$th component are adjacent and their signs are $\pm1$,
$\rho^q({w'_i}^\varepsilon)$ is obtained from $\rho^q(w_i^\varepsilon)$ by deleting some subwords $\rho^r(AB^{-1})$ and $\rho^r(BA^{-1})$ for some $2 \leq r \leq q$. 
Since $A$ and $B$ in the $k$th component are adjacent and their signs are 0,
\begin{align*}
\rho^r(AB^{-1}) &= \rho^r(A) \rho^r(B^{-1}) \\
&= \rho^{r-1}(B^0 x^{-1})A\rho^{r-1}(x B^0)\rho^{r-1}(x^{-1})B^{-1}\rho^{r-1}(x) \\
&= \rho^{r-1}(x^{-1})A\rho^{r-1}(x)\rho^{r-1}(x^{-1})B^{-1}\rho^{r-1}(x), 
\end{align*}
where $x$ is a signed word obtained by truncating $w_k^{\varepsilon }$ at $B$.
Since $\eta (A)= a_k$ and $\eta (B^{-1})= a_k^{-1}$, we have 
\begin{align*}
\eta ( \rho^r(AB^{-1}) ) &= \eta (\rho^{r-1}(x^{-1})A\rho^{r-1}(x)\rho^{r-1}(x^{-1})B^{-1}\rho^{r-1}(x))  \\
&= \eta (\rho^{r-1}(x^{-1})) a_k \eta (\rho^{r-1}(x)) \eta (\rho^{r-1}(x^{-1})) a_k^{-1} \eta (\rho^{r-1}(x)). 
\end{align*}
Therefore 
\begin{align*}
\varphi \circ \eta ( \rho^{r}(AB^{-1})) = 1.
\end{align*}
Similarly $\varphi \circ \eta ( \rho^{r}(BA^{-1})) = 1$.
Therefore $\varphi \circ \eta ( \rho^q (w_i^\varepsilon))$ is equal to $\varphi \circ \eta ( \rho^q ({w'_i}^\varepsilon))$.
Thus $\mu (p;c_1, \ldots ,c_{u},i)$ is equal to $\mu (p';c_1, \ldots ,c_{u},i)$.
\end{proof}

\begin{remark}
The moves on nanophrases corresponding to second Reidemeister moves on links are 
$(\mathcal{A}, xAByBAz) \longleftrightarrow (\mathcal{A}-\{ A, B\} , xyz)$ and
$(\mathcal{A}, xAByABz) \longleftrightarrow (\mathcal{A}-\{ A, B\} , xyz)$.
% phrases $... AB ... BA ...$ and $... AB ... AB ...$.     
But by Lemma 2.2 in \cite{MR2276346}, it is sufficient to show the first case. 
\end{remark}

\begin{proposition}\label{prop3}
If the indices $c_1, c_2, \ldots ,c_u,i$ are pairwise distinct integers between 1 and $n$, then $\mu (p;c_1, \ldots ,c_{u},i)=\mu (p';c_1, \ldots ,c_{u},i)$ if $p$ and $p'$ are related by an H3 move. 
\end{proposition}
\begin{proof}
Let us $p=xAByACzBCt$ and $p'=xBAyCAzCBt$, where $(|A|,|B|,|C|)$ is an element of $S$.
Moreover let us $|A|=|B|=|C|=b_+$.
(We can show the other cases in the same way.)
We need to consider the following four cases. 
\begin{description}
\item[{\rm Case 1}] The subwords $AB$, $AC$ and $BC$ appear in a certain component. \\
(i.e. $p$ is a nanophrase of the form  $... |... AB ... AC ... BC ...| ...$ ) 
\item[{\rm Case 2}] The subwords $AB$ and $AC$ appear in a certain component, but the subword $BC$ does not appear in that component. \\
(i.e. $p$ is a nanophrase of the form  $... |... AB ... AC ...|...|... BC ...|...$ ) 
\item[{\rm Case 3}] The subwords $AC$ and $BC$ appear in a certain component, but the subword $AB$ does not appear in that component. \\
(i.e. $p$ is a nanophrase of the form  $...|...AB ...|...|... AC ... BC ...| ...$ ) 
\item[{\rm Case 4}] The subwords $AB$, $AC$ and $BC$ appear in three different components. \\
(i.e. $p$ is a nanophrase of the form  $...|... AB ...|...|... AC ...|...|...BC ...| ...$) 
\end{description}  

Case 1: 
The letters $A$, $B$ and $C$ appear twice in a component and so they appear in neither ${w}_i^\varepsilon$ nor ${w'_i}^\varepsilon$.  
Hence they affect neither $\rho^q ({w}_i^\varepsilon )$ nor $\rho^q ({w'_i}^\varepsilon )$.
Thus $\rho^q ({w}_i^\varepsilon )$ is equal to $\rho^q ({w'_i}^\varepsilon )$. 
 
Case 2:
The letter $A$ appears twice in a component and hence it affects neither $\rho^q ({w}_i^\varepsilon )$ nor $\rho^q ({w'_i}^\varepsilon )$.
Since the letters $B$ and $C$ are adjacent in the component that contains the subword $BC$ (resp. $CB$) and the sign associated with these letters is 0, 
the difference between $BC$ and $CB$ does not affect $\rho^q ({w}_i^\varepsilon )$ and $\rho^q ({w'_i}^\varepsilon )$.
Thus $\rho^q ({w}_i^\varepsilon )$ is equal to $\rho^q ({w'_i}^\varepsilon )$. 
 
Case 3:
Let $h$ be the order of the component in which the subword $AB$ occurs 
and $k$ the order of the component in which the subwords $AC$ and $BC$ occur. 
By the same reason as in case 1, the letter $C$ affects neither $\rho ^q(w_i^\varepsilon)$ nor $\rho ^q({w'_i}^\varepsilon)$.
We divide the proof into the following two cases. 
\begin{description}
\item[{\rm Case 3-1}] $i=k$ 
\item[{\rm Case 3-2}] $i\neq k$ 
\end{description}

Case 3-1: 
Using Lemma $\ref{lem0}$, we only need to show that 
\begin{align*}
\#  \mathcal{S}_{\text{e}}({F}_k^q,c_1, \ldots,c_{u})- \# \mathcal{S}_{\text{o}}({F}_k^q,c_1, \ldots,c_{u}) \\
= \#  \mathcal{S}_{\text{e}}({F'_k}^q,c_1, \ldots,c_{u})- \# \mathcal{S}_{\text{o}}({F'_k}^q,c_1, \ldots,c_{u}), 
\end{align*}
where $F_k^q$ and ${F'_k}^q$ denote $F_k^q(p)$ and $F_k^q(p')$, respectively. 
Consider the subgraph of $F_k^q$ obtained by deleting the vertices labeled $k$ and their descendants, and denote it by ${F}_k^q \setminus \cup T_k$.
Since $i=k$, $k \notin \{c_1,...,c_u \}$ by assumption.
Thus we actually have the fact that the cardinality of $\mathcal{S}_\ast({F}_k^q,c_1, \ldots,c_{u})$ is equal to the cardinality of $\mathcal{S}_{\ast}({F}_k^q \setminus \cup T_k,c_1, \ldots,c_{u})$ for both $\ast =$ e and $\ast =$ o.
We define the subgraph ${F'_k}^q \setminus \cup T_k$ as above.
If we ignore the first labels, then the subgraphs ${F}_k^q \setminus \cup T_k$ and ${F'_k}^q \setminus \cup T_k$ are isomorphic as labeled forests.
Therefore we have the identity $\mu(p;c_1, \ldots ,c_{u},k) = \mu(p';c_1, \ldots ,c_{u},k)$.

Case 3-2:
If $k\notin \{c_1, \ldots,c_{u} \}$, it follows from the same reason as in case 3-1 that  $\mu(p;c_1, \ldots ,c_{u},i) = \mu(p';c_1, \ldots ,c_{u},i)$.
Let $k\in \{c_1, \ldots,c_{u} \}$.
Since the indices $c_1, \ldots,c_{u}$ are pairwise distinct and $\eta(A)=\eta(B)=a_k$, $\mathcal{S}_{\ast}({F}_i^q, c_1, \ldots,c_{u})$ is equal to
\begin{align*}
 \mathcal{S}_{\ast}({F}_i^q, c_1, \ldots,c_{u}, A) \sqcup \mathcal{S}_{\ast}({F}_i^q, c_1, \ldots,c_{u}, B) \sqcup \mathcal{S}_{\ast}({F}_i^q, c_1, \ldots,c_{u}, \bar{A}, \bar{B}), \notag 
\end{align*}
where $\mathcal{S}_{\ast}({F}_i^q, c_1, \ldots,c_{u}, A)$ (resp. $\mathcal{S}_{\ast}({F}_i^q, c_1, \ldots,c_{u}, B)$) denotes the set of elements of $\mathcal{S}_{\ast}({F}_i^q, c_1, \ldots,c_{u})$ with vertices labeled $A$ (resp. $B$) 
and  $\mathcal{S}_{\ast}({F}_i^q, c_1, \ldots,c_{u}, \bar{A}, \bar{B})$ denotes the set of  elements of $\mathcal{S}_{\ast}({F}_i^q, c_1, \ldots,c_{u})$ without vertices labeled $A$ and $B$.
If an element of $\mathcal{S}({F}_i^q, c_1, \ldots,c_{u})$ has a vertex labeled $A$, then the element does not have a vertex labeled $B$.
Therefore for both $\ast = $ e and o,
\begin{align*}
\mathcal{S}_{\ast}({F}_i^q, c_1, \ldots,c_{u}, A) 
= \mathcal{S}_{\ast}({F}_i^q \setminus \cup T_B,c_1, \ldots,c_{u}, A),
\end{align*}
where ${F}_i^q \setminus \cup T_B$ denotes the subgraph of $F_i^q$ obtained by deleting the vertices labeled $B$ and their descendants.
Similarly, 
\begin{align*}
\mathcal{S}_{\ast}({F}_i^q, c_1, \ldots,c_{u}, B) &= \mathcal{S}_{\ast}({F}_i^q \setminus \cup T_A,c_1, \ldots,c_{u}, B), \\ 
\mathcal{S}_{\ast}({F}_i^q, c_1, \ldots,c_{u}, \bar{A}, \bar{B}) &= \mathcal{S}_{\ast}({F}_i^q \setminus \cup T_A \cup T_B,c_1, \ldots,c_{u}, \bar{A}, \bar{B}). 
\end{align*}
If we ignore the first labels, then ${F}_i^q \setminus \cup T_A$ and ${F'_i}^q \setminus \cup T_A$ are isomorphic as labeled forests.
If we consider $\cup T_B$ or $\cup T_A \cup T_B$ instead of $\cup T_A$, then we obtain the same results. 
Thus $\mu(p;c_1, \ldots ,c_{u},i) $ is equal to $\mu(p';c_1, \ldots ,c_{u},i) $.

Case 4:
Let $h$ be the order of the component in which the subword $AB$ occurs, 
$j$ the order of the component in which the subword $AC$ occurs,
and $k$ the order of the component in which the subword $BC$ occurs. 
We divide the proof into the following three cases. 
\begin{description}
\item[{\rm Case 4-1}] $i=j$ 
\item[{\rm Case 4-2}] $i=k$ 
\item[{\rm Case 4-3}] $i\neq j,k$ 
\end{description}

Case 4-1:
Consider the subgraphs of $F_j^q$ and ${F'_j}^q$ obtained by deleting the vertices labeled $A$ and their descendants.
If we ignore the first labels, then they are isomorphic as labeled forests.
Since $j \notin \{c_1, \ldots,c_{u} \}$ and $\eta(A)=a_j$, it follows from case 3-1 that we have $\mu(p;c_1, \ldots ,c_{u},j)$ is equal to $\mu(p';c_1, \ldots ,c_{u},j)$. 

Case 4-2:
Consider the subgraphs of $F_i^q$ and ${F'_i}^q$ obtained by deleting the vertices labeled $B$ and $C$ and their descendants.
If we ignore the first label, then they are isomorphic as labeled forests.
Since $k\notin \{c_1, \ldots,c_{u} \}$ and $\eta(B)=\eta(C)=a_k$, we have the identity $\mu(p;c_1, \ldots ,c_{u},k) = \mu(p';c_1, \ldots ,c_{u},k)$. 

Case 4-3:
The associated forest $F_i^q$ is obtained from ${F'_i}^q$ by deleting those vertices labeled $C$ and their descendants which have a vertex labeled $A$ as parent,
and then by exchanging the vertices labeled $A$ and their descendants for the corresponding the vertices labeled $B$ and their descendants.
For both $\ast =$ e and $\ast =$ o, then we have
\begin{align*}
\mathcal{S}_{\ast}({F}_i^q, c_1, \ldots,c_{u}) = \mathcal{S}_{\ast}({F}_i^q, c_1, \ldots,c_{u}, A) \sqcup \mathcal{S}_{\ast}({F}_i^q, c_1, \ldots,c_{u}, \bar{A}). 
\end{align*}
Consider the subgraphs of $F_i^q$ and ${F'_i}^q$ obtained by deleting the vertices labeled $A$ and their descendants.
If we ignore the first label, then they are isomorphic as labeled forests.
Therefore we have 
\begin{align*}
\mathcal{S}_{\ast}({F}_i^q, c_1, \ldots,c_{u}, \bar{A}) = \mathcal{S}_{\ast}({F'_i}^q,c_1, \ldots,c_{u}, \bar{A}). 
\end{align*}
Therefore we only need to show that 
\begin{align}
\# \mathcal{S}_{\text{e}}({F}_i^q, c_1, \ldots,c_{u}, A) - \# \mathcal{S}_{\text{o}}({F}_i^q, c_1, \ldots,c_{u}, A) \notag \\
= \# \mathcal{S}_{\text{e}}({F'_i}^q,c_1, \ldots,c_{u}, A) - \# \mathcal{S}_{\text{e}}({F'_i}^q,c_1, \ldots,c_{u}, A). \notag 
\end{align}

Let $V_A$ denote the set of those vertices of the forest ${F}_i^q$ that are labeled $A$ and which do not have an ancestor labeled $A$.
Fix an element $v$ in $V_A$.
Then let $F_v$ denote the subgraph of $F_i^q$ obtained by deleting those vertices and their descendants which either are labeled $A$ or have depth $q-2$, with the exception of $v$. 
Then we have 
\begin{align}
\mathcal{S}_{\ast}({F}_i^q, c_1, \ldots,c_{u}, A) &= \bigsqcup_{v\in V_A} \mathcal{S}_{\ast}({F}_i^q, c_1, \ldots,c_{u}, v) \notag \\
&= \bigsqcup_{v\in V_A} \mathcal{S}_{\ast}({F_v},c_1, \ldots,c_{u}). \notag
\end{align}
Similarly let $V'_A$ denote the set of those vertices of the forest ${F'_i}^q$ that are labeled $A$ and which do not have an ancestor labeled $A$.
For each $v$ in $V_A$, we denote by $v'$ its image under a natural bijection from $V_A$ to $V'_A$.
Then let $F'_{v'}$ denote the subgraph of ${F'_i}^q$ as above.

It is sufficient to show that for any $v$ in $V_A$ 
\begin{align}
\# \mathcal{S}_{\text{e}}({F_v},c_1, \ldots,c_{u}) - \# \mathcal{S}_{\text{o}}({F_v},c_1, \ldots,c_{u}) = \# \mathcal{S}_{\text{e}}({F'_{v'}},c_1, \ldots,c_{u}) - \# \mathcal{S}_{\text{o}}({F'_{v'}},c_1, \ldots,c_{u}). \notag 
\end{align}
If the depth of $v$ is $q-2$, then those vertices of ${F_v}$ and ${F'_{v'}}$ which have depth less than $q-2$ do not contain a vertex labeled $A$, and those vertices of ${F_v}$ and ${F'_{v'}}$ which have depth $q-2$ do not contain a vertex labeled $B$.
Therefore we obtain that the subwords of $\rho^q(w_i^\varepsilon)$ and $\rho^q({w_i'}^\varepsilon)$ corresponding to $F_v$ and $F'_v$, respectively, are equal.
If  the depth of $v$ is $s<q-2$, then $v$, the vertex labeled $B$ paired with $v$ and their descendants in $F_v$ correspond to the subword $\rho^{q-s}(A)\rho^{q-s-1}(B)$ or $\rho^{q-s-1}(B^{-1})\rho^{q-s}(A^{-1})$.
Then 
\begin{align}
\rho^{q-s}(A)\rho^{q-s-1}(B) = \rho^{q-s-1}(x^{-1}) A \rho^{q-s-1}(x) \rho^{q-s-1}(B), \label{h3-1}
\end{align}
where $x$ is a signed word obtained by truncating $w_j^{\varepsilon }$ at $C$.
Moreover $v'$, the vertex labeled $B$ paired with $v'$ and their descendants in $F'_{v'}$ correspond to the subword $\rho^{q-s-1}(B)\rho^{q-s}(A)$ or $\rho^{q-s}(A^{-1})\rho^{q-s-1}(B^{-1})$. 
Then we have 
\begin{align}
\rho^{q-s-1}(B)\rho^{q-s}(A) = \rho^{q-s-1}(B)\rho^{q-s-1}(C^{-1})\rho^{q-s-1}(x^{-1}) A \rho^{q-s-1}(x)\rho^{q-s-1}(C). \label{h3-2}
\end{align}
Since $\eta ( \rho^r(B))=\eta ( \rho^r(C))$ for any $r$, the image of ($\ref{h3-1}$) and ($\ref{h3-2}$) under $\varphi \circ \eta$ are equal.
Similarly, the image of $\rho^{q-s-1}(B^{-1})\rho^{q-s}(A^{-1})$ and $\rho^{q-s}(A^{-1})\rho^{q-s-1}(B^{-1})$ under $\varphi \circ \eta$ are equal.
Therefore the image of the subword of $\rho^q(w_i^\varepsilon)$ corresponding to $F_v$ under $\varphi \circ \eta$
is equal to the image of the subword of $\rho^q({w'_i}^\varepsilon)$ corresponding to $F'_{v'}$ under $\varphi \circ \eta$.
\end{proof}

\begin{proposition}\label{prop5}
If the indices $c_1, c_2, \ldots ,c_u,i$ are a sequence of integers between 1 and $n$ (the indices are possibly repeating), then $\mu (p;c_1, \ldots ,c_{u},i)=\mu (p';c_1, \ldots ,c_{u},i)$ if $p$ and $p'$ are related by a self crossing move. 
\end{proposition}
\begin{proof}
Let us $p=...|...A...A...|...$ and $p'$ also have the same form.
Furthermore, if we write $|A|_p$ for $|A|$ in $p$ and $|A|_{p'}$ for $|A|$ in $p'$, then $|A|_{p'}$ equals $\sigma_v (|A|_p)$.
Since the letter $A$ appears twice in a component, $\rho ^q(w_i^\varepsilon)$ is equal to $\rho ^q({w'_i}^\varepsilon)$. 
\end{proof}

%%%%%%%%%%%%%%%%%%%%%%%%%%%%%%%%%%%%%%%%%%%%%%%%%%%%%%%%%%%%%%%%%%%%%%%%%%%%%%%%%%%%%%%%%%%%%%%%%%%%%%%%%%%%%%%%%%%%%%%%%%%%%%%%%%%%%%
\section{Proof of invariance for shift move}
We will show that $\bar{\mu} (p;c_1,c_2, \ldots ,c_{u},i)$ is an invariant under a shift move by induction on $u$.
We suppose that for each $u<q-1$, if the indices $c_1, c_2, \ldots ,c_u,i$ are pairwise distinct,
then $\bar\mu (p; c_1, c_2, \ldots ,c_u,i)$ is an invariant under $M$-homotopy.  
Moreover we suppose that if $u<q-1$ and the nanophrases $p$ and $p'$ are $M$-homotopic,
then $\bar\mu(p; c_1, c_2, \ldots ,c_u,i)$ is congruent to $\bar\mu(p'; c_1, c_2, \ldots ,c_u,i)$.
Therefore it follows from the definition of $\Delta $ that $\Delta(p; c_1, c_2, \ldots ,c_u,i)$ is congruent to $\Delta(p'; c_1, c_2, \ldots ,c_u,i)$ for $u<q$.
We prepare the following two lemmas for the next proposition.
They are due to Milnor \cite{MR0092150}.

\begin{lemma}[Milnor \cite{MR0092150}; (16), (19), (14)]\label{lem2}
Let $M$ be the set $\{ a_1, \cdots ,a_n \}$.
For any element $a_j$ in $M$, 
the coefficient of $\kappa _{c_1} \cdots \kappa _{c_{u}}$ in $\varphi (a_j \eta (\rho^s(w_j^\varepsilon)) a_j^{-1} \eta(\rho^s(w_j^\varepsilon))^{-1})$ is congruent to 0 modulo $\Delta (p;c_1, c_2, \ldots ,c_u,i)$.
\end{lemma}

\begin{lemma}[Milnor \cite{MR0092150}; (16), (17), (12)]\label{lem1}
For any word $x$ on $M \cup M^{-1}$,
the coefficient of $\kappa _{c_1} \kappa _{c_2}  \cdots \kappa _{c_u}$ in the image of $\eta(\rho ^q(w_i^\varepsilon))$ under $\varphi $ is equal to that of $x\eta(\rho ^q(w_i^\varepsilon))x^{-1}$ under $\varphi$ modulo $\Delta(p; c_1, c_2, \ldots ,c_u, i)$. 
% The coefficient of $\kappa _{c_1} \kappa _{c_2}  \cdots \kappa _{c_u}$ in the image of $\rho ^q(w_i^\varepsilon)$ under $\varphi \circ \eta$ is equal to that in the image of its conjugates under $\varphi \circ \eta$ modulo $\Delta(p; c_1, c_2, \ldots ,c_u, i)$. 
\end{lemma}

\begin{proposition}\label{prop4}
If the indices $c_1, c_2, \ldots ,c_u,i$ are pairwise distinct integers between 1 and $n$, then $\bar{\mu} (p;c_1, \ldots ,c_{u},i)= \bar{\mu} (p';c_1, \ldots ,c_{u},i)$ if $p$ and $p'$ are related by a shift move. 
\end{proposition}
\begin{proof}
Suppose that $p$ and $p'$ are related by the shift move on the $k$th component.
That is $p=w_1| \cdots |w_{k-1}|Ax|w_{k+1}| \cdots |w_n$ and $p'=w_1| \cdots |w_{k-1}|xA|w_{k+1}| \cdots |w_n$,
where $A$ is a letter in $\mathcal{A}$.
Furthermore, let us write $|A|_p$ for $|A|$ in $p$ and $|A|_{p'}$ for $|A|$ in $p'$.
If $x$ contains the letter $A$, then $|A|_{p'}$ equals $\nu_v (|A|_p)$.
Otherwise, $|A|_{p'}$ equals $|A|_p$. 
We need to consider the following three cases. 
\begin{description}
\item[{\rm Case 1}] $x$ contains the letter $A$ 
\item[{\rm Case 2}] $w_1\cdots w_{k-1}$ contains the letter $A$ and $|A| = b+$ or $a-$ \\
         (or $w_{k+1}\cdots w_n$ contains the letter $A$ and $|A| = a+$ or $b-$) 
\item[{\rm Case 3}] $w_1\cdots w_{k-1}$ contains the letter $A$ and $|A| = a+$ or $b-$ \\
         (or $w_{k+1}\cdots w_n$ contains the letter $A$ and $|A| = b+$ or $a-$) 
\end{description}  

Case 1: 
As in the proof of invariance under H1 move, $\rho ^q(w_i^\varepsilon)$ is equal to $\rho ^q({w'_i}^\varepsilon)$ and so $\mu(p,c_1,...,c_{q-1},i)$ is equal to $\mu(p',c_1,...,c_{q-1},i)$. 

Case 2: 
Let $|A|=b+$. 
(We can show the other cases in the same way.)
Let $k$ be the order of the component in which the shifted $A$ occurs and 
let $h$ the order of the component in which the other $A$ occurs. 
We divide the proof into the following two cases. 
\begin{description}
\item[{\rm Case 2-1}] $i = k$
\item[{\rm Case 2-2}] $i \neq k$ 
\end{description}

Case 2-1:
Consider the subgraph of ${F'_k}^q$ obtained  by deleting descendants of vertices labeled $A$, and denote it by $F_A'$. 
If we ignore the first label, then ${F}_k^q$ and $F_A'$ are isomorphic as forests.
Since $i=k$, $k \notin \{c_1,...c_{q-1} \}$.
Moreover $\eta (A) = a_k$.
Thus we actually have the fact that the cardinality of $\mathcal{S}_\ast({F'_k}^q,c_1, \ldots,c_{q-1})$ is equal to the cardinality of $\mathcal{S}_{\ast}(F'_A, c_1, \ldots,c_{q-1})$ for both $\ast =$ e and $\ast =$ o.
As in the proof of invariance under H2 move, $\mu(p';c_1, \ldots ,c_{q-1},k)$ is equal to $\mu(p;c_1, \ldots ,c_{q-1},k)$. 

Case 2-2:
The associated forest ${F}_i^q$ is obtained from ${F'_i}^q$ by deleting descendants of vertices labeled $A$.
For both $\ast =$ e and $\ast =$ o, we have the fact that $\mathcal{S}_{\ast}({F'_i}^q,c_1, \ldots,c_{q-1})$ is equal to 
\begin{align*}
 \mathcal{S}_{\ast}({F'_i}^q,c_1, \ldots,c_{q-1}, c(A)) \sqcup \mathcal{S}_{\ast}({F'_i}^q,c_1, \ldots,c_{q-1}, \overline{c(A)}),
\end{align*}
where $c(A)$ means the set of children of vertices labeled $A$.
Since $F'_A$ is isomorphic to ${F}_i^q$ as labeled forests,
\begin{align*}
\mathcal{S}_{\ast}({F'_i}^q,c_1, \ldots,c_{q-1}, \overline{c(A)}) &= \mathcal{S}_{\ast}({F}'_A,c_1, \ldots,c_{q-1}) \\
&= \mathcal{S}_{\ast}({F}_i^q,c_1, \ldots,c_{q-1}). 
\end{align*}
Therefore we only need to show that
\begin{align}
\# \mathcal{S}_{\text{e}}({F'_i}^q,c_1, \ldots,c_{q-1}, c(A)) &- \# \mathcal{S}_{\text{o}}({F'_i}^q,c_1, \ldots,c_{q-1}, c(A)) \label{2-ii-2} \\
& \equiv 0 \mod{\Delta (p';c_1, \ldots ,c_{q-1},i)}.  \notag
\end{align}

For any element $G$ in $\mathcal{S}({F'_i}^q,c_1, \ldots,c_{q-1}, c(A))$, 
let $\hat{G}$ denote the subgraph of $G$ obtained by deleting the descendants of the vertex labeled $A$. 
The left hand term of (\ref{2-ii-2}) is equal to 
\begin{align*}
 \sum_{F \in \{ \hat{G} \mid G  \in \mathcal{S}({F'_i}^q,c_1, \ldots,c_{q-1}, c(A)) \} } 
\Bigl( \# \{G & \in \mathcal{S}_{\text{e}}({F'_i}^q,c_1, \ldots,c_{q-1}, c(A)) \mid \hat{G}=F \} \\
 & - \# \{G\in \mathcal{S}_{\text{o}}({F'_i}^q,c_1, \ldots,c_{q-1}, c(A)) \mid \hat{G}=F \} \Bigr). 
\end{align*}
Let us fix $F$.
Let $r$ denote the depth of the vertex of $F$ labeled $A$.
Then we choose an element $G$ in $\mathcal{S}({F'_i}^q,c_1, \ldots,c_{q-1}, c(A))$ such that $\hat{G}$ is $F$.
Let $b_1,b_2,...,b_t$ be the subsequence of $c_1, \ldots,c_{q-1}$ corresponding to the vertex labeled $A$ and its descendants.
Let $s=q-r-1$, then we have 
\begin{align}
&\# \{G \in \mathcal{S}_{\text{e}}({F'_i}^q,c_1, \ldots,c_{q-1}, c(A)) \mid \hat{G}=F \} \notag \\
& \hspace{1.5cm} - \# \{G\in \mathcal{S}_{\text{o}}({F'_i}^q,c_1, \ldots,c_{q-1}, c(A)) \mid\hat{G}=F \}  \notag \\
&  = \varepsilon _F \cdot \Bigl( \# \mathcal{S}_{\text{e}}(T',b_1, \ldots, b_{t}) - \# \mathcal{S}_{\text{o}}(T',b_1, \ldots,b_{t}) \Bigr), \label{F0}
\end{align}
where $T'$ is either ${T'}_{A}^{s+1}$ or the tree obtained by changing the sign of the fourth label of the root in ${T'}_{A}^{s+1}$.
We then consider vertices of $F$ except the vertex labeled $A$.
If the number of those vertices which have $-1$ as the fourth label is even (resp. odd), then let $\varepsilon _F=1$ (resp. $\varepsilon _F=-1$).
By Lemma \ref{lem0}, the right hand term of (\ref{F0}) is equal to
\begin{align}
\varepsilon _F \cdot \Bigl(\text{the coefficient of $\kappa _{b_1} \cdots \kappa _{b_{t}}$ in $\varphi \bigl( \eta ( \rho^{s}(({w'_k}^\varepsilon)^{-1}))a^{\pm}_k \eta ( \rho^{s}({w'_k}^\varepsilon)) \bigr) $} \Bigr).  \label{F}
\end{align}
We know that $r+t-1\leq q-2$ and so $t\leq s$.
By Lemma \ref{lem2}, the coefficient of $\kappa _{b_1} \cdots \kappa _{b_{t}}$ in $\varphi \bigl( \eta ( \rho^{s}(({w'_k}^\varepsilon)^{-1}))a^{\pm}_k \eta ( \rho^{s}({w'_k}^\varepsilon)) \bigr)$ is equal to the coefficient of $\kappa _{b_1} \cdots \kappa _{b_{t}}$ in $\varphi (a_k^\pm )$ modulo $\Delta (p';b_1, b_2, \ldots ,b_t,i)$, and so $\Delta (p';c_1, \ldots ,c_{q-1},i)$.
Since $t\geq 2$ and $b_1,\ldots, b_t$ are pairwise distinct, the coefficient of $\kappa _{b_1} \cdots \kappa _{b_{t}}$ in $\varphi (a_k^\pm )$ is zero.
Therefore the term (\ref{F}) is equal to zero modulo $\Delta (p';c_1, \ldots ,c_{q-1},i)$.

Case 3: 
Let $|A|=a+$.
(We can show the other cases in the same way.)
Let $k$ be the order of the component in which the shifted $A$ occurs 
and let $h$ the order of the component in which the other $A$ occurs. 
We divide the proof into the following three cases. 
\begin{description}
\item[{\rm Case 3-1}] $i = h$ 
\item[{\rm Case 3-2}] $i = k$ 
\item[{\rm Case 3-3}] $i \neq h,k$ 
\end{description}

Case 3-1: 
As in the proof of invariance under H3 move, ${\mu}(p';c_1, \ldots ,c_{q-1},h)$ is equal to ${\mu}(p;c_1, \ldots ,c_{q-1},h)$.

Case 3-2: 
We can represent $\rho^q(w_k^\varepsilon)$ by $\rho^q(A)x$, 
where $x$ is some word on $\mathcal{A} \cup \mathcal{A}^{-1}$.
It follows from Lemma \ref{lem1} that the coefficient of $\kappa _{c_1} \cdots \kappa _{c_{q-1}}$ in $\varphi \circ \eta (\rho^q(A)x)$ 
is equal to that in $\varphi \circ \eta (x\rho^q(A))$ modulo $\Delta(p;c_1, \ldots,c_{q-1},k)$.
Therefore let $\bar{F}_k^q$ denote the forest associated with a word $x\rho^q(A)$, and we have 
\begin{align*}
\mu(p;c_1, \ldots ,c_{q-1},k)
& \equiv \# \mathcal{S}_{\text{e}}(\bar{F}_k^q,c_1, \ldots,c_{q-1}) - \# \mathcal{S}_{\text{o}}(\bar{F}_k^q,c_1, \ldots,c_{q-1}) \\
& \mod{\Delta (p;c_1, \ldots ,c_{q-1},k)}. 
\end{align*}
If we ignore the first labels,
the subgraphs obtained from $\bar{F}_k^q$ and ${F'_k}^q$ by deleting the vertices labeled $k$ and their descendants are isomorphic as labeled forests.
Therefore we have 
\[ \mu(p;c_1, \ldots ,c_{q-1},k) \equiv \mu(p';c_1, \ldots ,c_{q-1},k) \mod{\Delta (p;c_1, \ldots ,c_{q-1},k)}. \]
Thus $\bar{\mu}(p;c_1, \ldots ,c_{q-1},k)$ is equal to $\bar{\mu}(p';c_1, \ldots ,c_{q-1},k)$.

Case 3-3: 
Since ${F'_i}^q$ is isomorphic to the subgraph of ${F}_i^q$ obtained by deleting the vertices labeled $A$ and their descendants, 
as in the proof of Proposition \ref{prop3} case 4-3, we only need to show that  
\begin{align}
\begin{split} 
\# \mathcal{S}_{\text{e}}({F}_i^q,c_1, \ldots,c_{q-1}, A) & - \# \mathcal{S}_{\text{o}}({F}_i^q,c_1, \ldots,c_{q-1}, A)  \\
 & \equiv 0 \mod \Delta (p;c_1,...c_{q-1},i) \label{(A,a_h)2}  
\end{split}
\end{align}
Since $i\neq k$, vertices labeled $A$ have a vertex labeled $k$ as their parent.
Denote by $\mathcal{C}$ the set of all sequences obtained from $c_1, \ldots, c_{q-1}$ by deleting $h$ and possibly some other indices.
% We consider the set of all sequences obtained from $c_1, \ldots, c_{q-1}$ by deleting some vertices 
% where we must delete indices $h$, and denote it by $\mathcal{C}$.
Let $A_\ast(b_1,...,b_r) $ denote 
\[ \{G \in \mathcal{S}_{\ast}({F}_i^q,c_1, \ldots,c_{q-1},A) \mid G \setminus \cup{T_A} \in \mathcal{S}({F}_i^q,b_1, \ldots,b_{r}) \} \] 
for both $\ast =$ e and $\ast =$ o.
Then we have
\begin{align}
\mathcal{S}_{\ast}({F}_i^q,c_1, \ldots,c_{q-1}, A) = \bigsqcup_{(b_1,...,b_r)\in \mathcal{C}} A_\ast(b_1,...,b_r). \notag
\end{align}
We will momentarily show that 
\begin{align}
A_\ast(b_1,...,b_r) &= \delta _{b_1,...,b_r} \cdot \Bigl( \# \mathcal{S}_{\text{e}}({F}_i^q,b_1, \ldots,b_{r}) - \# \mathcal{S}_{\text{o}}({F}_i^q,b_1, \ldots,b_{r}) \Bigr) \label{2.11} \\
&= \delta _{b_1,...,b_r} \cdot \mu(p;b_1, \ldots ,b_r,i), \notag
\end{align}
where $\delta _{b_1,...,b_r}$ is an integer defined using the sequence $b_1,...,b_r$ (see below). 
Therefore since for any $(b_1,...,b_r)$ in $\mathcal{C}$,  $\mu(p;b_1, \ldots ,b_r,i)$ is divisible by $\Delta (p;c_1,...c_{q-1},i)$, Equation (\ref{(A,a_h)2}) is proved.

We now prove Equation (\ref{2.11}).
Let $B_\ast(b_1,...,b_r)$ denote $\mathcal{S}_{\ast}({F}_i^q,b_1, \ldots,b_{r})$ for both $\ast =$ e and $\ast =$ o.
If $A_\text{e} \sqcup A_\text{o} = \emptyset $, then we set $\delta _{b_1,...,b_r}= 0$. 
If not, we define a map $\theta : A_\text{e} \sqcup A_\text{o}  \longrightarrow B_\text{e} \sqcup B_\text{o} $ by $\theta (g) = \hat{g} $ for any $g \in A_\text{e} \sqcup A_\text{o} $.
Then $\theta $ is surjective. 
In fact, if there exists $h$ in $B_\text{e} \sqcup B_\text{o} $ such that for any $g$ in $A_\text{e} \sqcup A_\text{o}$ the image of $g$ under $\theta$ is not equal to $h$, then $A_\text{e} \sqcup A_\text{o} = \emptyset $.
In addition, for any $e \in B_{\ast} $, 
$\# (\theta ^{-1}(e) \cap A_{\ast} )$ has the same value, denoted by $m_{\ast}$, for both $\ast =$ e and $\ast =$ o.
Moreover for any $e \in B_\text{e}$ and $e' \in B_\text{o} $, 
\begin{eqnarray*}
\# (\theta ^{-1}(e) \cap A_\text{e} ) = \# (\theta ^{-1}(e') \cap A_\text{o} ) = m_\text{e} \\
\# (\theta ^{-1}(e) \cap A_\text{o} ) = \# (\theta ^{-1}(e') \cap A_\text{e} ) = m_\text{o}.
\end{eqnarray*}
\begin{eqnarray*}
A_\text{e} &=& \bigsqcup_{e \in B_\text{e}} (\theta ^{-1}(e) \cap A_\text{e} ) \bigsqcup_{e' \in B_\text{o}} (\theta ^{-1}(e') \cap A_\text{e} ), \\
A_\text{o} &=& \bigsqcup_{e \in B_\text{e}} (\theta ^{-1}(e) \cap A_\text{o} ) \bigsqcup_{e' \in B_\text{o}} (\theta ^{-1}(e') \cap A_\text{o} )
\end{eqnarray*}
show
\begin{eqnarray*}
\#A_\text{e} &=& \#B_\text{e} m _\text{e} + \#B_\text{o} m _\text{o}, \\
\#A_\text{o} &=& \#B_\text{e} m _\text{o} + \#B_\text{o} m _\text{e}.
\end{eqnarray*}
Therefore
\begin{eqnarray*}
\#A_\text{e} - \#A_\text{o} 
&=& (\#B_\text{e} m _\text{e} + \#B_\text{o} m _\text{o}) - ( \#B_\text{e} m _\text{o} + \#B_\text{o} m _\text{e} ) \\
&=& (m _\text{e} - m _\text{o})(\#B_\text{e} - \#B_\text{o}).
\end{eqnarray*}
Hence we may set $\delta _{b_1,...,b_r}$ as $m _\text{e} - m _\text{o}$.
\end{proof}

\begin{remark}
The shift move on nanophrases corresponds to a change of base point on links.  
Therefore the invariance under the shift move corresponds to (12), (13) in Theorem 5 in \cite{MR0092150}.
\end{remark}

%%%%%%%%%%%%%%%%%%%%%%%%%%%%%%%%%%%%%%%%%%%%%%%%%%%%%%%%%%%%%%%%%%%%%%%%%%%%%%%%%%%%%%%%%%%%%%%%%%%%%%%%%%%%%%%%
\section{Welded links}
Fix a finite set $\alpha $ and choose an element which is not contained in $\alpha$, which we denote by $\emptyset$.
Let $\bar\alpha $ be the union of $\alpha $ and $\{\emptyset\}$.
Then we extend a projection from $\mathcal{A}$ to $\alpha $ to a map from the union of $\mathcal{A}$ and the empty word $\emptyset$ to $\bar\alpha $ which is defined by $|\emptyset|=\emptyset$.
Let $S'$ be a subset of $\bar\alpha \times \bar\alpha \times \bar\alpha$.  
We also call the triple $(\alpha , \tau , S')$ a {\it homotopy} data.

Fixing $\alpha$ and $S'$, we define an {\it extended H3 move} on nanophrases over $\alpha$ as follows.
The move is

\hspace{1.5em} Extended H3 move : if $(|A|, |B|, |C|) \in S'$,

\hspace{6.5em}$(\mathcal{A}, xAByACzBCt) \longleftrightarrow (\mathcal{A}, xBAyCAzCBt)$. 

We define an extended $M$-homotopy to be the equivalence relation of nanophrases over $\alpha $ generated by isomorphisms, 
the three usual homotopy moves H1 - H3 together with the extended H3 move with respect to $(\alpha ,\tau ,S')$,  
self crossing moves with respect to $\sigma$ and shift moves with respect to $\nu$.

Here we recall that two virtual link diagrams are said to be {\it welded equivalent} if one may be transformed into the other by a sequence of generalized Reidemeister moves and upper forbidden moves in Fig.\ \ref{fig:8}.
A {\it welded link} can then be defined to be a welded equivalence class of virtual link diagrams. 

We below show that there exists a homotopy data $(\alpha, \tau, S')$ and $\nu$ corresponding to the welded equivalence relation.

\begin{figure}[H]
\begin{center}
\includegraphics[scale=0.4, angle=0]{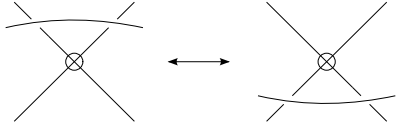}
\end{center}
\caption{Upper forbidden move} \label{fig:8}
\end{figure}

\begin{theorem}
Let $\alpha_{v} $ be the set $\{a_+,a_-,b_+,b_-\}$ and $\tau _{v} $ the involution on $\alpha_{v} $ which sends $a_+$ to $b_-$ and $a_-$ to $b_+$.
Let $\nu_{v} $ be the involution on $\alpha_{v} $ which sends $a_+$ to $b_+$ and $a_-$ to $b_-$.
Let $S'_{w} $ be the set 

\hspace{5.5em}$S'_{w}  = \left\{
\begin{array}{l}
{(a_\pm,a_\pm,a_\pm),(a_\pm,a_\pm,a_\mp),(a_\pm,a_\mp,a_\mp), }\\  
   {(b_\pm,b_\pm,b_\pm),(b_\pm,b_\pm,b_\mp),(b_\pm,b_\mp,b_\mp), } \\  
  {(a_+,a_+,\emptyset),(a_+,b_-,\emptyset),(b_-,a_+,\emptyset),(b_-,b_-,\emptyset),}\\  
   {(a_-,\emptyset,b_-),(b_+,\emptyset,a_+),(a_-,\emptyset,a_+),(b_+,\emptyset,b_-),}\\  
   {(\emptyset,a_-,a_-),(\emptyset,a_-,b_+),(\emptyset,b_+,a_-),(\emptyset,b_+,b_+)}  
\end{array}
\right\}$. \\
Under the homotopy defined by $(\alpha_{v}, \tau_{v}, S'_{w})$ and $\nu_v$, the set of homotopy classes of nanophrases over $\alpha_{v}$ is in a bijective correspondence with the set of ordered welded links. 
\end{theorem}

\begin{proof}
By Theorem \ref{e.thm}, the set of homotopy classes of nanophrases over $\alpha_{v} $ under the homotopy with respect to $(\alpha_{v}  ,\tau_{v}  ,S_{v})$ and $\nu_v$ is in a bijective correspondence with the set of virtual links. 
Consider the upper forbidden move. 
It suffices to consider the case when the orientation of the three arcs are as illustrated in Fig.\ \ref{fig:5}.
(The deformations involving other orientations of the arcs can be obtained as compositions of this one with isotopy and local deformations of the second Reidemeister moves.)  
There are 6 cases to consider depending on the order in which one traverses the three arcs involved.
Let the order of arcs be as illustrated in Fig.\ \ref{fig:5}.
Then this move transforms the associated phrase from $xAyBzABt$ to $xAyBzBAt$, where $x, y, z$ and $t$ are words not including the letters $A$ and $B$, and $|A|=|B|=a_-$.
Conversely if two nanophrases are represented by $xAyBzABt$ and $xAyBzBAt$, the associated virtual link diagrams are related by upper forbidden moves.
We can show the other cases in the same way.
Therefore under the homotopy defined by $(\alpha_{v}  ,\tau_{v}  ,S'_w)$ and $\nu_v$, the set of homotopy classes of nanophrases over $\alpha_{v}$ is in a bijective correspondence with ordered welded links. 
\end{proof}

\begin{figure}[H]
\begin{center}
\includegraphics[scale=0.4]{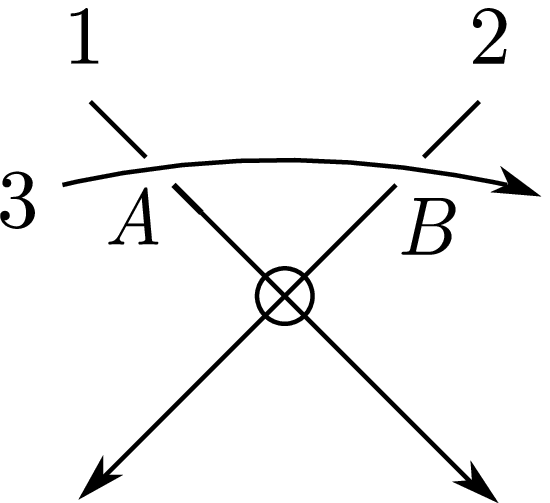}
\end{center}
\caption{} \label{fig:5}
\end{figure}

\begin{theorem} \label{them2}
Let $p$ be an $n$-component nanophrase.
Let $c_1, c_2, \ldots ,c_u,i$ be a sequence of integers between $1$ and $n$ such that $c_1, c_2, \ldots ,c_u,i$ are pairwise distinct.
Then $\bar{\mu} (p; c_1, c_2, \ldots ,c_u,i)$ 
is an invariant under $M$-homotopy of nanophrases with respect to $(\alpha_v, \tau_v, S'_w)$ associated with upper forbidden moves, $\nu_v$ and $\sigma_v$.
\end{theorem}

\begin{proof}
By the proof of Theorem \ref{m.them}, it suffices to show that $\mu (p;c_1, \ldots ,c_{u},i)=\mu (p';c_1, \ldots ,c_{u},i)$ if $p$ and $p'$ are related by an extended H3 move corresponding to an upper forbidden move. 
Let us $p=xAByACzBCt$ and $p'=xBAyCAzBCt$, where $(|A|,|B|,|C|)$ is an element of $S'_w$.
Moreover let us $|A|=|B|=b_-$ and $|C|=\emptyset$. 
That is $p=xAByAzBt$ and $p'=xBAyAzBt$, where $|A|=|B|=b_-$.
(We can show the other cases in the same way.)
Since $\eta(A)=\eta(B)$, Theorem \ref{them2} follows from the same arguments as in proof of Proposition \ref{prop3}. 
\end{proof}

Note that two nanowords are always equivalent under $M$-homotopy defined by $(\alpha_{v}  ,\tau_{v}  ,S'_w)$, $\nu_v$ and $\sigma _v$ associated with the welded equivalence relation.   

%%%%%%%%%%%%%%%%%%%%%%%%%%%%%%%%%%%%%%%%%%%%%%%%%%%%%%%%%%%%%%
\section*{Acknowledgement}
The author thanks Professor Sadayoshi Kojima for his valuable
suggestions and comments.
She would like to thank Professor Michael Polyak for his correspondence regarding his result and giving many valuable comments. 
Also, she would like to thank Dr.\ Tomonori Fukunaga for many helpful comments. 
Finally, she would like to thank Professor Tamas Kalman for many valuable comments. 
\bibliography{kotorii}
\bibliographystyle{amsplain}

\end{document}